\documentclass[12pt]{amsart}

\usepackage[margin=1in]{geometry}  
\usepackage{graphicx}              
\usepackage{amsmath}               
\usepackage{amsfonts}              
\usepackage{amsthm}                
\usepackage{amsmath, amssymb}
\usepackage[all]{xy}
\usepackage{enumerate}
\usepackage{graphicx}
\usepackage{color}
\usepackage{subcaption}

\newtheorem{thm}{Theorem}[section]
\newtheorem{lem}[thm]{Lemma}

\newtheorem{cor}[thm]{Corollary}
\newtheorem{conj}[thm]{Conjecture}

\theoremstyle{definition}
\newtheorem{remark}[thm]{Remark}

\theoremstyle{definition}
\newtheorem{defn}[thm]{Definition}

\theoremstyle{definition}
\newtheorem{exmp}[thm]{Example}

\newcommand{\Le}{\textup{\protect\scalebox{-1}[1]{L}}}

\begin{document}

\title{Green-to-red Sequences for Positroids} 
\author{Nicolas Ford}
\address{Department of Mathematics, University of California, Berkeley, 
CA 94720-3840, USA}
\email{nicf@math.berkeley.edu}

\author{Khrystyna Serhiyenko}\thanks{The second author was  supported by the NSF Postdoctoral fellowship MSPRF-1502881.}
\address{Department of Mathematics, University of California, Berkeley, 
CA 94720-3840, USA}
\email{khrystyna.serhiyenko@berkeley.edu}

\begin{abstract}
\Le-diagrams are combinatorial objects that parametrize cells of the totally nonnegative Grassmannian, called positroid cells, and each \Le-diagram gives rise to a cluster algebra which is believed to be isomorphic to the coordinate ring of the corresponding positroid variety.  We study quivers arising from these diagrams and show that they can be constructed from the well-behaved quivers associated to Grassmannians by deleting and merging certain vertices.  Then, we prove that quivers coming from arbitrary \Le-diagrams, and more generally reduced plabic graphs, admit a particular sequence of mutations called a green-to-red sequence.      
\end{abstract}

\maketitle

\section{Introduction}

Since their introduction in 2001, a lot of attention has been paid to Fomin and Zelevinsky's cluster algebras, a class of commutative rings with a distinguished set of generators which are related to each other by a collection of relatively simple combinatorial rules.  Many previously well-studied rings, including Grassmannians and double Bruhat cells in flag varieties \cite{S,FZ3}, turn out to be cluster algebras.  Cluster algebras come with a great deal of extra structure, such as a canonical basis and a well-behaved notion of positivity.

A part of the combinatorial framework of many cluster algebras is encoded in an operation called \emph{quiver mutation}, which is a simple transformation performed on a directed graph.  In the cases where it is relevant, numerous properties of the cluster algebra can be read off from the quiver alone.  For example, skew-symmetric cluster algebras of finite type and finite mutation type are classified via their quivers, see \cite{FST, FZ2}.

Another important characteristic of a quiver is the existence of a particular sequence of mutations, called a \emph{green-to-red sequence}.  It was introduced in \cite{K}, as a combinatorial tool for computing the refined Donaldson-Thomas invariants. The existence of such sequence has important consequences \cite[Sec. 4.4]{M} for the structure of the \emph{scattering diagram}, a recent construction appearing in \cite{GHKK} which was instrumental in resolving several long-standing conjectures in the field.  In particular, it follows that the Enough Global Monomials property of \cite{GHKK} holds whenever these sequences exist.  
Furthermore, in all known cases quivers admitting a {\it maximal green sequence}, a particular type of a green-to-red sequence, also have the property that their cluster algebra agrees with their upper cluster algebra.  See \cite{CLS} for a detailed overview, however a precise relationship between the two concepts remains unclear.  Later, it was shown in \cite{M} that the existence of such sequence is not invariant under mutation, even though the cluster algebra remains the same. On the other hand, the existence of a green-to-red sequence is preserved under mutation \cite{M} which offers a more precise characterization of the corresponding cluster algebra.  

Next, we briefly recall some of the relevant geometric objects. Given two nonnegative integers $d\le n$ and a field $k$, write $Gr_{d,n}(k)$ for the Grassmannian of $d$-planes in $k^n$. Given a $d$-plane $V\subseteq k^n$, we may represent it as an $n\times d$ matrix $M$ by choosing any linear map $k^d\to k^n$ whose image is $V$. Recall that, given any $d$-element subset $S$ of the set $\{1,\ldots,n\}$, we may define the corresponding \emph{Pl\"ucker coordinate} by taking the determinant of the submatrix of $M$ containing only the rows in $S$. Choosing a different $M$ to represent the same $d$-plane $V$ amounts to multiplying $M$ on the right by an element of $GL_d$, which in turn multiplies all the Pl\"ucker coordinates by the same nonzero scalar. The Pl\"ucker coordinates thereby comprise a set of projective coordinates on the Grassmannian itself.

Given any point in the Grassmannian, we say the \emph{matroid} of that point is the set of Pl\"ucker coordinates that are nonzero, and the corresponding \emph{open matroid variety} is the set of all points of the Grassmannian with the same matroid as the chosen point. The closure of the open matroid variety is just called the \emph{matroid variety}. 

When $k=\mathbb R$, we can speak of the \emph{totally non-negative Grassmannian}, which is the set of all points of $Gr_{d,n}(\mathbb R)$ for which all Pl\"ucker coordinates are nonnegative. The special matroid of a point in the totally non-negative Grassmannian is called a \emph{positroid}, and the corresponding matroid variety is called a \emph{positroid variety}. For more information on these varieties, see \cite{KLS}.

The cluster algebras, we discuss in this paper, arise from certain combinatorial objects and they are believed to give a cluster structure to positroid varieties.  

\begin{conj}[\cite{MuSp}, 3.4]\label{*}
The coordinate ring of an open positroid variety has a cluster structure obtained from the quiver associated to any of its reduced plabic graphs.
\end{conj}

The conjecture is shown to be true in the case of Grassmannians \cite{S}, the top-dimensional positroid varieties.  More precisely, each positroid gives rise to a family of reduced plabic graphs, which are all mutation equivalent, and exactly one of those graphs corresponds to a \Le-diagram \cite{P}. We show that quivers arising from arbitrary \Le-diagrams can be obtained by merging and deleting certain vertices of quivers associated to Grassmannians.  An explicit green-to-red sequence for the latter set of quivers is constructed in \cite{MS}, and we use its existence to prove the following.

\begin{thm}
Quivers associated to reduced plabic graphs admit a green-to-red sequence. 
\end{thm}

In particular, this implies that the Enough Global Monomials property holds in this case.  We also obtain another class of examples of cluster algebras, that agree with their upper cluster algebras, as follows from \cite{M2} and \cite{MuSp}, and accept a green-to-red sequence.  Combining these results with those in \cite{GHKK} we can derive the following conclusion.  If Conjecture \ref{*} holds then the coordinate ring of an open positroid variety has a canonical basis of theta functions, parametrized by the lattice of $g$-vectors.  This in turn raises a number of interesting questions regarding the explicit computation of such basis and its structure constants.

In Section 2, we review the combinatorics of \Le-diagrams and the associated quivers. In Section 3, we provide an explicit construction of quivers arising from \Le-diagrams.  We use this description to prove the existence of a green-to-red sequence for such quivers in Section~4.

\medskip
\noindent{\bf Acknowledgments:} The authors would like to thank Lauren Williams for bringing this project to their attention, Greg Muller for several fruitful conversations, and Steven Karp for valuable comments on the exposition. 
\section{Preliminaries}

\subsection{Quiver mutation and green-to-red sequences}

In this section we review the basic facts and definitions about quiver mutations that we will use in our construction.

A \emph{quiver} is a finite directed graph with no loops ($\xymatrix{\bullet \ar@(ru,rd)}$ \text{    }\text{        }\text{  }) and no directed 2-cycles ($\xymatrix{\bullet \ar@/^/[r] &\bullet \ar@/^/[l]}$). Next, consider the following definition introduced in \cite{FZ}.

\begin{defn}
Given a quiver $Q$ and a vertex $k$, we define a new quiver $\mu_k(Q)$, which we refer to as the result of \emph{mutating} $Q$ at $k$. We form $\mu_k(Q)$ by performing the following three operations in order:

\begin{enumerate}[(1)]
\item For each pair of arrows passing through $k$, of the form $i\to k\to j$, add a new arrow $i\to j$.
\item Reverse the direction of every arrow incident to $k$.
\item Remove all directed 2-cycles.
\end{enumerate}

Mutating at a fixed vertex is an involution. Accordingly, we say that two quivers are \emph{mutation-equivalent} if there is a sequence of mutations taking one to the other.

In some contexts, we will want to disallow mutation at some of the vertices of a quiver. When we do this, we say those vertices are \emph{frozen}. A frozen vertex remains frozen after any sequence of mutations.
\end{defn}

Note that the result of mutating a quiver does not depend on any edges between two frozen vertices. Because of this, it is customary to ignore such edges.

The central result of this paper is the existence of a certain sequence of mutations called a green-to-red sequence, also known as a \emph{reddening sequence}. We present the definition here for completeness, but as we will see, the details will not be necessary for the proof of the main theorem. This exposition closely follows \cite{M}, which the reader should consult for more details.

\begin{defn}
A \emph{$g$-seed} is a quiver $Q$ with $n$ vertices together with, for each vertex of $Q$, an element of the lattice $\mathbb{Z}^n$, with the property that the chosen vectors form a basis. We refer to these as the \emph{$g$-vectors} of the corresponding vertices.

Given an ordering of the vertices of $Q$, there is an obvious $g$-seed formed by assigning the $i$'th vertex the vector $e_i$, which has an entry 1 in the $i$'th coordinate and 0's elsewhere. We call this the \emph{initial $g$-seed}.

Given a $g$-seed and an element $v\in\mathbb{Z}^n$, we can write \[v=\sum_{k\in Q}c_k(v)g_k,\] where $g_k$ is the $g$-vector of vertex $k$. In this notation, we say that a nonfrozen vertex $k$ is \emph{green} if, for each $v\in\mathbb{N}^n$, $c_k(v)\ge 0$, and we say it is \emph{red} if each such $c_k(v)\le 0$. Observe, that every nonfrozen vertex in the initial seed is green.

Like an ordinary quiver, a $g$-seed can be \emph{mutated} to produce a new $g$-seed. The underlying quiver mutates in the same way as above. If we are mutating at vertex $k$, all $g$-vectors other than $g_k$ remain unchanged, and we define
\[g_k'=
\begin{cases}
-g_k+\sum_{\mathrm{arrows}\ j\to k}g_j & k\mbox{ is green} \\
-g_k+\sum_{\mathrm{arrows}\ j\gets k}g_j & k\mbox{ is red}
\end{cases}
\]
\end{defn}

Remarkably, if we start with the initial seed and perform an arbitrary sequence of mutations, each nonfrozen vertex of the resulting $g$-seed is either green or red, which is necessary for this mutation procedure to be well-defined. (This is the \emph{Sign Coherence Theorem}, proved in \cite{DWZ}.) In fact, mutation turns green vertices red and vice-versa, and it is straightforward to check that this means it remains an involution.

\begin{defn}
Given a quiver $Q$, a \emph{green-to-red sequence} is a sequence of mutations of $Q$ starting at the initial seed and ending at a seed in which all non-frozen vertices are red. If each mutation in the sequence occurs at a green vertex, we call it a \emph{maximal green sequence}. 
\end{defn}

Note that the existence of a green-to-red sequence or a maximal green sequence does not depend on the frozen vertices of the quiver. That is, it suffices to construct one after removing all frozen vertices and any incident edges.

Maximal green sequences first appeared in an application of this machinery to physics \cite{K}, but they were soon seen to exist for large classes of quivers coming from cluster algebras that were known to be well-behaved, see for instance \cite{A, BDP, BHIT, GY, Mi}. It was conjectured that various desirable properties of cluster algebras are equivalent to the existence of a maximal green sequence.

It turns out, though, that the weaker notion of green-to-red sequence suffices for certain applications, and in fact they turn out to be better behaved. In particular, we will take advantages of the following two facts, proved in \cite{M}.

\begin{thm}[\cite{M}, 3.1.3]\label{gtrdelete}
If a quiver $Q$ admits a green-to-red sequence, then so does any quiver obtained from $Q$ by deleting a set of vertices and their adjacent edges.
\end{thm}

\begin{thm}[\cite{M}, 3.2.2]\label{gtrmutate}
If a quiver $Q$ admits a green-to-red sequence, then so does any quiver mutation-equivalent to $Q$.
\end{thm}

While the first of these facts is true of maximal green sequences, the second is not. (A counterexample can be found in \cite{M}.) We will, in fact, not construct a green-to-red sequence directly for the class of quivers we are interested in.  Instead, we will use Theorems \ref{gtrdelete} and \ref{gtrmutate} to deduce its existence from the fact that for a smaller class of quivers such a sequence was previously constructed in \cite{MS}. (In fact, they produce a maximal green sequence.) While it is possible to use the arguments here to explicitly produce the resulting sequence of mutations, we do not perform that computation in the present paper.

\subsection{\Le-diagrams, $\Gamma$-graphs, and plabic graphs}
Following the exposition in \cite{P}, we describe bijections between the various combinatorial objects that are involved in our construction.  First, we show how to pass from a \Le-diagram to the associated $\Gamma$-graph, and afterwards to the corresponding plabic graph.  Finally, we construct a quiver that can be thought of as a dual graph of a given plabic graph.  

\begin{defn}
Given a partition $\lambda$, a \emph{\Le-diagram} of shape $\lambda$ is a filling of the boxes of the Young diagram of $\lambda$ with 0's and 1's such that, whenever a box in the $i$-th row and $j$-th column contains a 0, either all boxes in $i$-th row to the left of it or all boxes in the $j$'th column above it are filled with 0's.  
\end{defn}

Figure~\ref{fig:Le} (left) shows an example of a \Le-diagram.  Here dots in the boxes of the Young diagram indicate that they are filled with 1's and empty boxes are filled with 0's.  If we draw two lines going to the right and down from the dotted box, the \Le-property means that every box located at the intersection of two lines contains a dot.

\begin{figure}[htb]
  \centering
  \includegraphics[scale=.3]{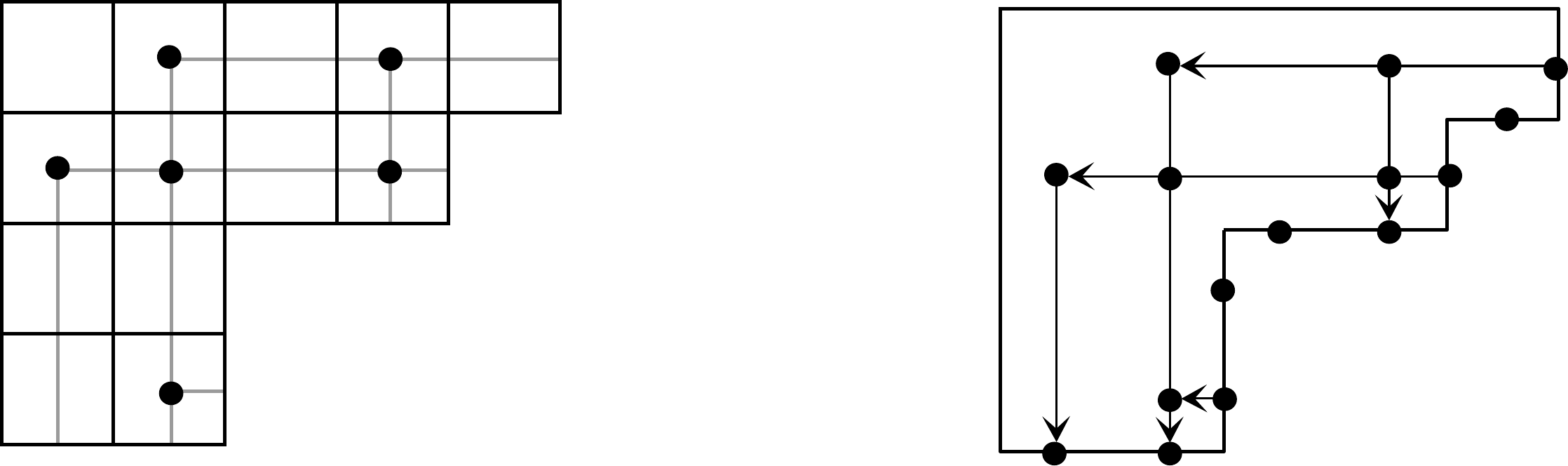}
  \caption{A \Le-diagram of shape $\lambda = (5,4,2,2)$ (left) and the corresponding $\Gamma$-graph (right). }
   \label{fig:Le}
\end{figure}

\begin{defn}
A \emph{$\Gamma$-graph} is a planar directed graph $\mathcal{G}$ satisfying the following conditions: 
\begin{enumerate}[(1)]
\item The graph $\mathcal{G}$ is drawn inside a closed boundary curve in $\mathbb{R}^2$, with some \emph{internal} vertices inside the curve and some \emph{boundary} vertices along the edge. 
\item $\mathcal{G}$ contains only vertical edges oriented downward and horizontal edges oriented to the left.
\item For any internal vertex $x$, the graph $\mathcal{G}$ contains the line going down from $x$ until it hits the boundary (at some boundary sink vertex) and the line going to the right from $x$ until it hits the boundary (at some boundary source vertex). 
\item Each pairwise intersection of such lines is also a vertex of $\mathcal{G}$. 
\item The graph may also contain some number of isolated boundary vertices, which are assigned to be sinks or sources. 
\end{enumerate}
\end{defn}

Given a \Le-diagram $D$ we can construct the corresponding $\Gamma$-graph $\mathcal{G}_D$ in the following way.  Suppose $D$ has shape $\lambda$, where $\lambda$ fits inside an $n\times k$ rectangle.  The boundary of the Young diagram of $\lambda$ gives a lattice path of length $n+k$ from the upper right corner to the lower left corner of the rectangle.  Place a vertex in the middle of each step in this lattice path.  The vertices corresponding to the vertical steps in the lattice path will be sources and the vertices corresponding to horizontal steps will be sinks.  Then connect the upper right corner with the lower left corner of the rectangle by another path so that, together with the lattice path, it forms a closed curve containing the Young diagram in its interior.  Place an internal vertex in each box of $D$ containing a 1.  Finally, draw the line that goes downward from each such internal vertex vertex until it hits a boundary sink and another line that goes to the right from this vertex until it hits a boundary source.  For an example, see Figure~\ref{fig:Le}.

\begin{defn}
A \emph{plabic graph} is a finite planar graph drawn inside a disk where each internal vertex is colored either black or white and each boundary vertex is incident to a single edge.  We consider these graphs up to homotopy. 
\end{defn}

There is a set of transformations studied in \cite{P} that can be performed on a plabic graph resulting in a new plabic graph while preserving some of its properties.  In particular,  we are allowed to remove/insert interior vertices of any color of degree 2.  That is, given a vertex of degree 2 as follows $\xymatrix@C=10pt{\ar@{-}[r]&*=+{\circ}\ar@{-}[r]&}$ we can remove it from the graph and glue the adjacent edges to obtain $\xymatrix@C=10pt{\ar@{-}[rr]&&}$.  Conversely,  given an edge we can insert an interior vertex in the middle.  In this paper we will be working with equivalence classes of plabic graphs related in this way.

Given a $\Gamma$-graph $\mathcal{G}$ we can construct the corresponding plabic graph by applying the transformation in Figure~\ref{fig:transformation} to each internal vertex of $\mathcal{G}$. The boundary of $\mathcal{G}$ is mapped to the boundary of the plabic graph.  

\begin{remark}\label{internalvertex}
Let $x$ be an internal vertex in a $\Gamma$-graph $\mathcal{G}$.   By definition there exists a vertical arrow starting at $x$ and ending at some boundary sink, and a horizontal arrow ending at $x$ and starting at some boundary source.  Therefore, the result of transformation described in Figure~\ref{fig:transformation} in a local neighborhood of $x$ depends only on the existence of two internal vertices in $\mathcal{G}$, one vertex located above $x$ and the other vertex located to the left of $x$.
\end{remark}

\begin{figure}

\[
\hspace*{3cm}\xymatrix@R=10pt@C=15pt{&&&&&&&&&\ar[dd]&&&&&\ar[dd]&&&&\\
*=0{\bullet}\ar[d]&\ar[l]&&&&*=0{\bullet} \ar[d] & \ar[ll] &&& *=0{\bullet}&\ar[l] &&&& *=0{\bullet} & \ar[ll]\\
&&&&&&&&&&&&&&&&&\\
\ar@{.>}[d]&&&&&\ar@{.>}[d]&&&&\ar@{.>}[d]&&&&&\ar@{.>}[d]\\
&&&&&&&&&&&&&&&\\
&&&&&&&&&\ar@{-}[dd]&&&&& \ar@{-}[d]\\
*=0{\bullet}\ar@{-}[d]&\ar@{-}[l]&&&\ar@{-}[r]&*=+{\circ}\ar@{-}[r]\ar@{-}[d] &&&&*=0{\bullet}&\ar@{-}[l]&&&&*=0{\bullet}&\ar@{-}[l]\\
&&&&&&&&&&&&&*=+{\circ}\ar@{-}[ur]\ar@{-}[d]\ar@{-}[l]&&\\
&&&&&&&&&&&&&&&}
\]

\caption{Transformation of a $\Gamma$-graph to the corresponding plabic graph.}
\label{fig:transformation} 
\end{figure}
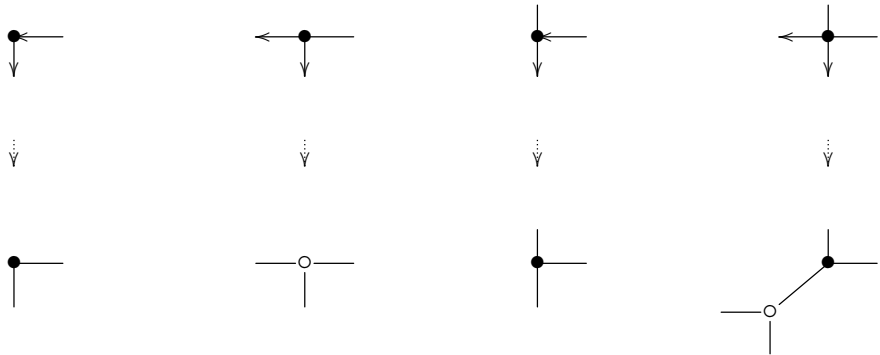

Given a plabic graph $G$ we construct the corresponding quiver $Q$.  The vertices of $Q$ correspond to faces of $G$, and for every edge $e_{ij}$ between a black and a white vertex, separating distinct faces labeled $F_i, F_j$, there exists an arrow $v_i \to v_j$ in $Q$ oriented so that the black vertex is to the left of the arrow.  The vertices of $Q$ that come from a face adjacent to the boundary of the plabic graph are frozen, so we remove any arrows between such vertices.  We also remove a maximal collection of oriented 2-cycles. 

If $\lambda$ is a partition of zero we say that the corresponding plabic graph is a disk with no vertices, so the corresponding quiver consists of a single frozen vertex and no arrows.

\begin{figure}[htb]
  \centering
  \includegraphics[scale=.5]{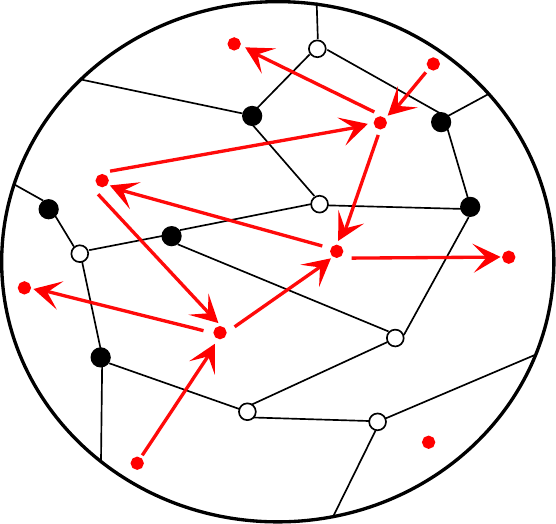}
  \caption{Plabic graph (black) and the corresponding quiver (red).}
   \label{fig:plabic}
\end{figure}

\begin{remark}
The quiver associated to a plabic graph is not affected by applying homotopy relations or removing/inserting vertices of degree 2.  Therefore, the above construction is well-defined.   
\end{remark}

Plabic graphs obtained from \Le-diagrams are said to be reduced, in the sense of \cite{P}.  Reduced plabic graphs are considered up to certain move equivalences, and precisely one such graph in its equivalence class comes from a \Le-diagram.  Moreover, whenever two such graphs are equivalent then their corresponding quivers are related by a sequence of mutations.  Hence, to prove the main theorem it suffices to construct green-to-red sequences for quivers arising from \Le-diagrams.  Note that equivalence classes of reduced plabic graphs are actually in bijection with \Le-diagrams, see \cite{P} for more details.     

\section{Quivers arising from \Le-diagrams}

Quivers arising from \Le-diagrams associated to rectangular Young diagrams such that every box is filled with 1 have a well-understood, regular structure.   They have been studied in \cite{S}.  In this section, we describe quivers that arise from arbitrary \Le-diagrams.  

A \Le-diagram $D$ of a rectangular shape $n\times k$ filled with $1$'s corresponds to a quiver $Q_D$ consisting of $nk+1$ vertices, where $(n-1)(k-1)$ vertices are nonfrozen and $n+k$ vertices are frozen.   If $n=0$ or $k=0$ then we say that $Q_D$ consists of a single frozen vertex.  Label the vertices of $Q_D$ in the following way.  To each box $b_{i,j}$ of $D$ we associate a vertex $v_{i,j}$, such that the southeast corner of $b_{i,j}$ belongs to a face $F_{i,j}$ in the corresponding plabic graph $G_D$. In this way a vertex $v_{i,j}$ is nonfrozen if and only if $1\leq i \leq k-1$ and $1\leq j \leq n-1$.  Moreover, $Q_D$ contains one additional frozen vertex, which we label $v_0$.  We can think of $v_0$ as the face $F_0$ of $G_D$ containing the northwest corner of the box $b_{1,1}$.  See Figure~\ref{fig:grass} for an example of such labeling.  Here the gray lines outline the Young diagram, and the outer boundary of the Young diagram corresponds to the outer boundary of $G_D$.  The frozen vertices are colored blue.

\begin{figure}[htb]
\hspace*{8cm}{\xymatrix@C=15pt@R=15pt{{\color{blue} v_0} &&&&&& \\
v_{1,1}\ar[u] \ar[dr]&v_{1,2} \ar[l]\ar[dr]&v_{1,3}\ar[l]\ar[dr]&v_{1,4}\ar[l]\ar[dr]&{\color{blue} v_{1,5}}\ar[l]\\
  v_{2,1}\ar[u]\ar[dr] & v_{2,2}\ar[l]\ar[u]\ar[dr] & v_{2,3}\ar[l]\ar[u]\ar[dr] & v_{2,4}\ar[l]\ar[u]\ar[dr] & {\color{blue} v_{2,5}}\ar[l]\\
   v_{3,1}\ar[u]\ar[dr] & v_{3,2}\ar[l]\ar[u]\ar[dr] & v_{3,3}\ar[l]\ar[u]\ar[dr] & v_{3,4}\ar[l]\ar[u]\ar[dr] & {\color{blue} v_{3,5}}\ar[l]\\
   {\color{blue} v_{4,1}}\ar[u] & {\color{blue} v_{4,2}} \ar[u] & {\color{blue} v_{4,3}}\ar[u] & {\color{blue} v_{4,4}}\ar[u] & {\color{blue} v_{4,5}} }}


  \vspace*{-4cm}\hspace*{-10cm}{\includegraphics[scale=.7]{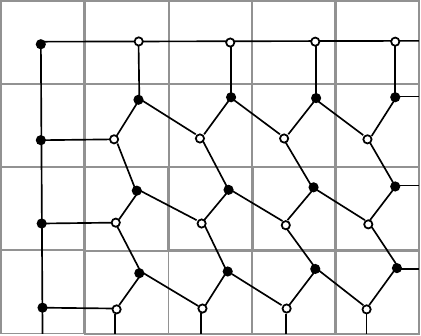}}
  \caption{Plabic graph (left) corresponding to a rectangular $4\times5$ \Le-diagram filled only with 1's and the corresponding quiver (right).}
   \label{fig:grass}
\end{figure}

\begin{lem}\label{lemma1}
Let $\mu, \lambda$ be two partitions with $\mu\subset \lambda$ and $\lambda$ having rectangular shape.  Let $D_{\lambda}$ and $D_{\mu}$ be the corresponding \Le-diagrams where every box is filled with 1.  Let $G_{\lambda}, G_{\mu}$ be the associated plabic graphs, and let $Q_{\lambda}, Q_{\mu}$ be the associated quivers.
\medskip

\begin{enumerate}[(a)]

\item $G_{\mu}$ is obtained from $G_{\lambda}$ by removing the faces of $G_\lambda$ corresponding to boxes in $\lambda\setminus \mu$.  

\item $Q_{\mu}$ is obtained from $Q_{\lambda}$ by removing the vertices in $Q_\lambda$ corresponding to boxes in $\lambda\setminus\mu$ and the arrows between the resulting frozen vertices of $Q_{\mu}$.
\end{enumerate}
\end{lem}

\begin{proof}

(a).  Observe that we can obtain $D_{\mu}$ from $D_{\lambda}$ by removing one box at a time such that at each step we have a \Le-diagram. Therefore, it suffices to show that whenever we have two partitions $u'\subset u$ differing by a single box $b_{i,j}$, then the associated plabic graph $G_{u'}$ can be obtained from $G_{u}$ by removing the corresponding face $F_{i,j}$.  Note that, here we assume that every box in the associated \Le-diagrams is filled with 1.   

If $u$ consists of a single box, then $G_u$ has two faces $F_0$ and $F_{1,1}$.  Removing $F_{1,1}$ yields a plabic graph with one face and no internal vertices.  By our convention this graph corresponds to the partition of zero, so the statement holds.  If $u$ consists of more than one box, we consider three distinct cases based on the location of the face $F_{i,j}$.  In case 1, the corresponding box $b_{i,j}$ is neither in the first row nor in the first column of $D_{u}$. In case 2, it is at the bottom of the first column of $D_{u}$, and in case 3 it is on the right side of the first row of $D_u$.  The diagrams in the first column in Figure~\ref{fig:lemma1} illustrates the three resulting local configurations of the face $F_{i,j}$ in the graph $G_{u}$.  The black lines form a part of $G_{u}$, while the grey lines outline a part of the box $b_{i,j}$.   In the second column we show plabic graphs obtained from $D_u$ after removing the face $F_{i,j}$.  The graphs in the last column are homotopic to the ones in the second column, and we claim that these are precisely the plabic graphs associated to $G_{u'}$.  This will prove part (a) because plabic graphs are considered up to homotopy.  

Let $\mathcal{G}_u$, $\mathcal{G}_{u'}$ be the $\Gamma$-graphs associated to $D_u$, $D_{u'}$ as above.  Hence, $\mathcal{G}_u$ and $\mathcal{G}_{u'}$ differ by a single internal vertex $x_{i,j}$.   Moreover, by construction there are no other internal vertices in $\mathcal{G}_u$ below or to the right of $x_{i,j}$.  By Remark~\ref{internalvertex}, the presence of $x_{i,j}$ in $\mathcal{G}_u$ does not affect the transformation of other vertices in $\mathcal{G}_{u}$ or edges that are not attached to $x_{i,j}$.  Therefore, the plabic graphs $G_{u}$ and $G_{u'}$ are the same apart from the internal vertices arising from $x_{i,j}$ and their adjacent edges going to the right and down.  This shows that the plabic graphs in the third column in Figure~\ref{fig:lemma1} coincide with $G_{u'}$ and proves the claim.  

\begin{figure}[htb]
  \centering

 \begingroup%
  \def\svgwidth{400pt}
  \makeatletter%
  \providecommand\color[2][]{%
    \errmessage{(Inkscape) Color is used for the text in Inkscape, but the package 'color.sty' is not loaded}%
    \renewcommand\color[2][]{}%
  }%
  \providecommand\transparent[1]{%
    \errmessage{(Inkscape) Transparency is used (non-zero) for the text in Inkscape, but the package 'transparent.sty' is not loaded}%
    \renewcommand\transparent[1]{}%
  }%
  \providecommand\rotatebox[2]{#2}%
  \ifx\svgwidth\undefined%
    \setlength{\unitlength}{639.584008bp}%
    \ifx\svgscale\undefined%
      \relax%
    \else%
      \setlength{\unitlength}{\unitlength * \real{\svgscale}}%
    \fi%
  \else%
    \setlength{\unitlength}{\svgwidth}%
  \fi%
  \global\let\svgwidth\undefined%
  \global\let\svgscale\undefined%
  \makeatother%
  \begin{picture}(1,0.54641875)%
    \put(0,0){\includegraphics[width=\unitlength]{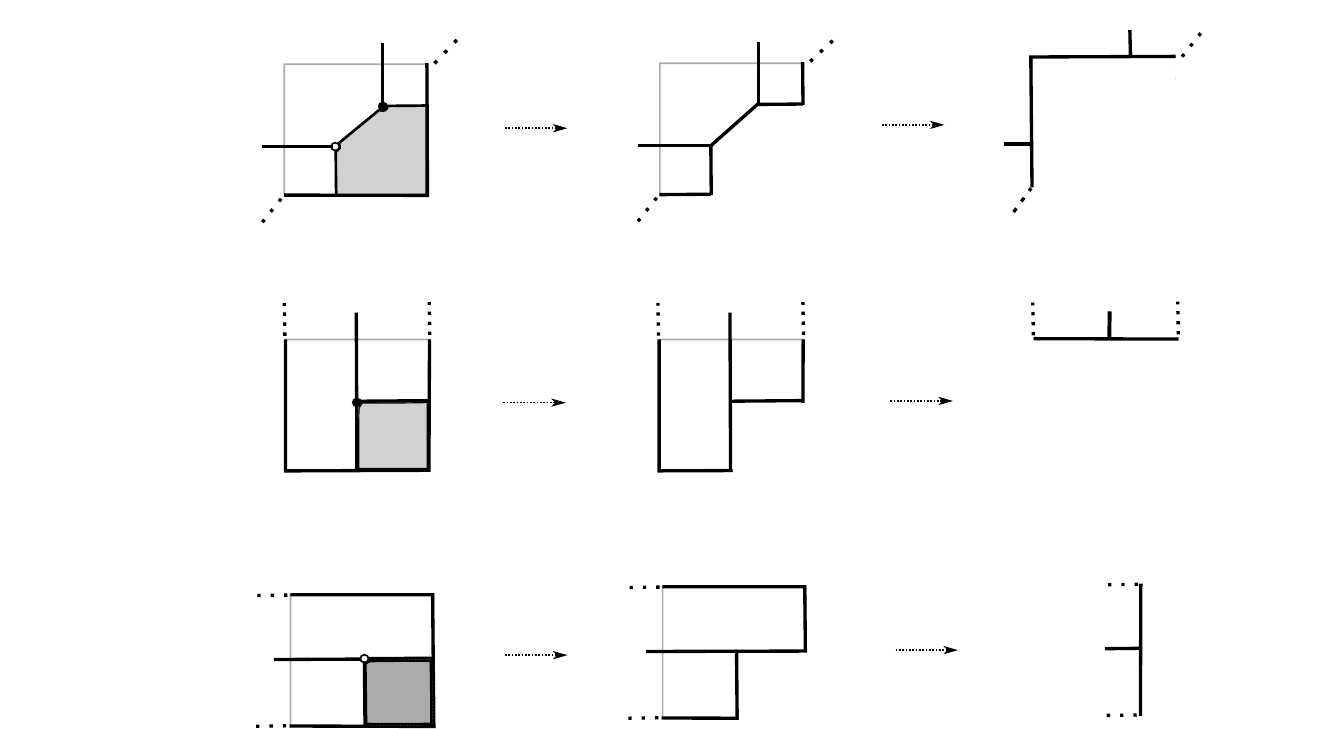}}%
    \put(0.27883324,0.4361914){\color[rgb]{0,0,0}\makebox(0,0)[lt]{\begin{minipage}{0.15635163\unitlength}\raggedright $\scriptstyle{F_{i,j}}$\end{minipage}}}%
    \put(0.27875692,0.22702478){\color[rgb]{0,0,0}\makebox(0,0)[lt]{\begin{minipage}{0.15635163\unitlength}\raggedright $\scriptstyle{F_{i,1}}$\end{minipage}}}%
    \put(0.27760178,0.03482725){\color[rgb]{0,0,0}\makebox(0,0)[lt]{\begin{minipage}{0.15635163\unitlength}\raggedright $\scriptstyle{F_{1,j}}$\end{minipage}}}%
    \put(0.19606359,0.42994357){\color[rgb]{0,0,0}\makebox(0,0)[lt]{\begin{minipage}{0.15635163\unitlength}\raggedright $\scriptstyle{F_{i,j\text{-}1}}$\end{minipage}}}%
    \put(0.47444702,0.43020525){\color[rgb]{0,0,0}\makebox(0,0)[lt]{\begin{minipage}{0.15635163\unitlength}\raggedright $\scriptstyle{F_{i,j\text{-}1}}$\end{minipage}}}%
    \put(0.71741722,0.42772313){\color[rgb]{0,0,0}\makebox(0,0)[lt]{\begin{minipage}{0.15635163\unitlength}\raggedright $\scriptstyle{F_{i,j\text{-}1}}$\end{minipage}}}%
    \put(0.22900559,0.26932294){\color[rgb]{0,0,0}\makebox(0,0)[lt]{\begin{minipage}{0.15635163\unitlength}\raggedright $\scriptstyle{F_0}$\end{minipage}}}%
    \put(0.50418445,0.27200326){\color[rgb]{0,0,0}\makebox(0,0)[lt]{\begin{minipage}{0.15635163\unitlength}\raggedright $\scriptstyle{F_0}$\end{minipage}}}%
    \put(0.78829769,0.31846203){\color[rgb]{0,0,0}\makebox(0,0)[lt]{\begin{minipage}{0.15635163\unitlength}\raggedright $\scriptstyle{F_0}$\end{minipage}}}%
    \put(0.21729566,0.03040496){\color[rgb]{0,0,0}\makebox(0,0)[lt]{\begin{minipage}{0.15635163\unitlength}\raggedright $\scriptstyle{F_{1,j\text{-}1}}$\end{minipage}}}%
    \put(0.49754171,0.03792252){\color[rgb]{0,0,0}\makebox(0,0)[lt]{\begin{minipage}{0.15635163\unitlength}\raggedright $\scriptstyle{F_{1,j\text{-}1}}$\end{minipage}}}%
    \put(0.8009854,0.0427149){\color[rgb]{0,0,0}\makebox(0,0)[lt]{\begin{minipage}{0.15635163\unitlength}\raggedright $\scriptstyle{F_{1,j\text{-}1}}$\end{minipage}}}%
    \put(0.52384007,0.08795506){\color[rgb]{0,0,0}\makebox(0,0)[lt]{\begin{minipage}{0.15635163\unitlength}\raggedright $\scriptstyle{F_0}$\end{minipage}}}%
    \put(0.25176171,0.0843813){\color[rgb]{0,0,0}\makebox(0,0)[lt]{\begin{minipage}{0.15635163\unitlength}\raggedright $\scriptstyle{F_0}$\end{minipage}}}%
    \put(0.81998822,0.09063536){\color[rgb]{0,0,0}\makebox(0,0)[lt]{\begin{minipage}{0.15635163\unitlength}\raggedright $\scriptstyle{F_0}$\end{minipage}}}%
    \put(0.55108971,0.2884552){\color[rgb]{0,0,0}\makebox(0,0)[lt]{\begin{minipage}{0.15635163\unitlength}\raggedright $\scriptstyle{F_{i\text{-}1,1}}$\end{minipage}}}%
    \put(0.83326279,0.31846203){\color[rgb]{0,0,0}\makebox(0,0)[lt]{\begin{minipage}{0.15635163\unitlength}\raggedright $\scriptstyle{F_{i\text{-}1,1}}$\end{minipage}}}%
    \put(0.27107358,0.28414131){\color[rgb]{0,0,0}\makebox(0,0)[lt]{\begin{minipage}{0.19299347\unitlength}\raggedright $\scriptstyle{F_{i\text{-}1,1}}$\end{minipage}}}%
    \put(0.21578883,0.48596849){\color[rgb]{0,0,0}\makebox(0,0)[lt]{\begin{minipage}{0.15635163\unitlength}\raggedright $\scriptstyle{F_{i\text{-}1,j\text{-}1}}$\end{minipage}}}%
    \put(0.49659,0.4871236){\color[rgb]{0,0,0}\makebox(0,0)[lt]{\begin{minipage}{0.15635163\unitlength}\raggedright $\scriptstyle{F_{i\text{-}1,j\text{-}1}}$\end{minipage}}}%
    \put(0.71762064,0.52806845){\color[rgb]{0,0,0}\makebox(0,0)[lt]{\begin{minipage}{0.15635163\unitlength}\raggedright $\scriptstyle{F_{i\text{-}1,j\text{-}1}}$\end{minipage}}}%
    \put(0.28564887,0.53605679){\color[rgb]{0,0,0}\makebox(0,0)[lt]{\begin{minipage}{0.15635163\unitlength}\raggedright $\scriptstyle{F_{i\text{-}1,j}}$\end{minipage}}}%
    \put(0.56829519,0.53505495){\color[rgb]{0,0,0}\makebox(0,0)[lt]{\begin{minipage}{0.15635163\unitlength}\raggedright $\scriptstyle{F_{i\text{-}1,j}}$\end{minipage}}}%
    \put(0.84872631,0.54878941){\color[rgb]{0,0,0}\makebox(0,0)[lt]{\begin{minipage}{0.15635163\unitlength}\raggedright $\scriptstyle{F_{i\text{-}1,j}}$\end{minipage}}}%
    \put(0.00448658,0.45607698){\color[rgb]{0,0,0}\makebox(0,0)[lt]{\begin{minipage}{0.15127991\unitlength}\raggedright Case 1\\ $\scriptstyle{i\not=1,\;j\not=1}$\end{minipage}}}%
    \put(-0.00056747,0.24842933){\color[rgb]{0,0,0}\makebox(0,0)[lt]{\begin{minipage}{0.15311169\unitlength}\raggedright Case 2\\ $\scriptstyle{i\not=1,\;j=1}$\end{minipage}}}%
    \put(0.00322307,0.0526898){\color[rgb]{0,0,0}\makebox(0,0)[lt]{\begin{minipage}{0.1479701\unitlength}\raggedright Case 3\\ $\scriptstyle{i=1,\;j\not=1}$\end{minipage}}}%
  \end{picture}%
\endgroup%

  \caption{Plabic graphs obtained as a result of removing the face $F_{i,j}$}
   \label{fig:lemma1}
\end{figure}
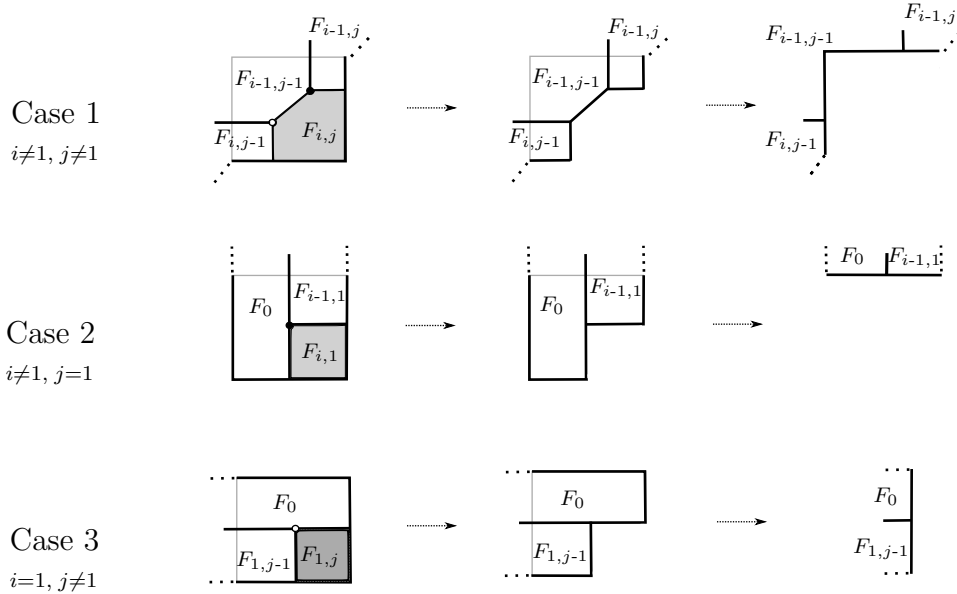

(b). By the same argument as above, it suffices to show that the statement holds whenever we have two partitions $u, u'$ as in part (a).  By part (a) we know that the plabic graphs $G_{u'}$ and $G_{u'}$ differ by a face $F_{i,j}$, so by construction the quiver $Q_{u'}$ is obtained from $Q_{u}$ by removing the vertex $v_{i,j}$.  Indeed, in cases 2 and 3 in Figure~\ref{fig:lemma1} we see that $v_{i,j}$ is a frozen vertex of $Q_u$ with no incoming or outgoing arrows.  Moreover, with the exception of $v_{i,j}$, the set of frozen vertices of $Q_{u'}$ coincides with that of $Q_{u}$.  This shows that the statement holds in cases 2 and 3.  In case 1, in addition to removing $v_{i,j}$, the vertex $v_{i-1,j-1}$ which was mutable in $Q_u$ becomes a frozen vertex in $Q_{u'}$.  Also, there are arrows $v_ {i,j-1}\to v_{i-1,j-1}$ and $v_{i-1,j}\to v_{i-1,j-1}$ in $Q_{u}$ which are no longer present in $Q_{u'}$.   Hence, in this case we also delete arrows between the new frozen vertices of $Q_{u'}$.  This shows part (b).  

\end{proof}


Next we consider how the quiver changes if we remove all edges between two adjacent faces of the corresponding plabic graph.

\begin{defn}
Given a set of vertices $S$ in a quiver $Q$, we say $Q'$ is obtained from $Q$ by \emph{merging} $S$ if it is the result of replacing all vertices of $S$ by a single vertex $v_{S}$ and, for every arrow starting (respectively ending) at some $v\in S$, drawing a new arrow starting (respectively ending) at $v_{S}$.  The vertex $v_S$ is frozen if there exists at least one frozen vertex in $S$.  Finally, we remove a maximal collection of oriented 2-cycles, loops, and arrows between the new frozen vertices created in this way.    
\end{defn}

\begin{lem}\label{lemma0}
Consider two plabic graphs $G, G'$ and their associated quivers $Q, Q'$.  Let $F_i, F_j$ be two distinct faces in $G$ that share an edge.  If $G'$ is obtained from $G$ by removing all edges between $F_i$ and $F_j$, then $Q'$ is obtained from $Q$ by merging the corresponding vertices $v_i$ and $v_j$.    
\end{lem}

\begin{proof}
Let $F_{ij}$ be the face in $G'$ obtained by combining $F_i$ and $F_j$ after removing edges between them.  By construction the vertices of $Q'$ are obtained from the vertices of $Q$ by removing $v_i, v_j$ and adding a new vertex $v_{ij}$.  If $F_i$ or $F_j$ is adjacent to the boundary, then $F_{ij}$ must also be adjacent to the boundary.  Hence, if $v_i$ or $v_j$ is frozen then so is $v_{ij}$.  Every edge in $G$ between $F_i$ or $F_j$ and some face $F_x$ where $x\not= i , j$, remains as an edge between $F_{ij}$ and $F_x$ in $G'$.   Conversely, every edge between $F_{ij}$ and $F_x$ in $G'$ arises from an edge in $G$ either between $F_{i}$ and $F_{x}$ or between $F_j$ and $F_x$.  This implies that the arrows in $Q$ starting or ending at $v_i$ or $v_j$ correspond to arrows starting or ending at $v_{ij}$ in $Q'$ up to removing 2-cycles.  Therefore, $Q'$ is obtained from $Q$ by merging $v_i$ and $v_j$.    
\end{proof}

Let $D_{\lambda}$ be some \Le-diagram filled with 0's and 1's corresponding to a partition $\lambda$.  As in Lemma \ref{lemma1}, to each box $b_{i,j}$ in $D_{\lambda}$ we associate a face $F_{i,j}$ in the corresponding plabic graph $G_{\lambda}$ and a vertex $v_{i,j}$ in the corresponding quiver $Q_{\lambda}$.  Let $b_{i,j}$ be a box in $D_{\lambda}$ filled with 0.  Then by definition exactly one of the following statements holds.

\begin{description}

\item [Case 1] all boxes in row $j$ and to the left of $b_{i,j}$ are filled with 0's but there exists a box in column $i$ above $b_{i,j}$ filled with 1.

\item [Case 2] all boxes in $D_{\lambda}$ in column $i$ and above $b_{i,j}$ are filled with 0's but there exists a box in row $j$ to the left of $b_{i,j}$ filled with 1.  

\item [Case 3] all boxes in column $i$ and above $b_{i,j}$ and all boxed in row $j$ and to the left of $b_{i,j}$ are filled with 0's. 
\end{description}

In the first case, let $b_{i'', j}$ be the box filled with 1 such that $1\leq i''<i$ and $i''$ is maximal and define $\mathcal{S}_{i,j}= \{F_{i',j}\mid i''\leq i \leq i\}$ to be the associated subset of faces of $G_{\lambda}$.  Similarly, in the second case let $b_{i, j''}$ be the box filled with $1$ such that $1\leq j''<j$ and $j''$ is maximal and define $\mathcal{S}_{i,j}=\{F_{i, j'}\mid j''\leq j' \leq j\}$. Finally, if all boxes to the left and above $b_{i,j}$ are filled with 0's then define $\mathcal{S}_{i,j}=\{F_{i',j'}\mid i'\leq i \text{ and } j'=j \text{ or } i'=i\text{ and } j'\leq j \}\cup\{F_0\}$.  Analogously, for every box $b_{i,j}$ filled with 0 we associate a subset of vertices $S_{i,j} = \{v_{i,j}\mid F_{i,j}\in \mathcal{S}_{i,j}\}$ in the quiver $Q_{\lambda}$. 

\begin{exmp}\label{S}
If $D_{\lambda}$ is given in Figure~\ref{fig:Le} then we have the following sets.  
$$\xymatrix@R=0pt{S_{1,1}=\{v_0, v_{1,1} \} && S_{2,3}= \{v_{2,2}, v_{2,3}\} && S_{4,1}=\{v_{2,1}, v_{3,1}, v_{4,1}\}\\
S_{1,3}=\{v_{1,2}, v_{1,3}\} && S_{3,1}=\{v_{2,1}, v_{3,1}\} && \\
S_{1,5}=\{ v_{1,4}, v_{1,5}\} && S_{3,2} = \{v_{2,2},v_{3,2}\} }$$  
\end{exmp}

\begin{lem}\label{lemma2}
Let $D_{\lambda}$ be an arbitrary \Le-diagram of shape $\lambda$, and let $D_\lambda'$ be the \Le-diagram of shape $\lambda$ filled entirely with 1's.   Let $G_{\lambda}, G_{\lambda}'$ denote the associated plabic graphs, and $Q_{\lambda}, Q_{\lambda}'$ denote the associated quivers.
\medskip

\begin{enumerate}[(a)]

\item $G_{\lambda}$ is obtained from $G_{\lambda}'$ by removing all edges between the faces in $\mathcal{S}_{i,j}$ for every box $b_{i,j}$ of $D_{\lambda}$ filled with 0.  

\item $Q_{\lambda}$ is obtained from $Q_{\lambda}'$ by merging the vertices in $S_{i,j}$ for every box $b_{i,j}$ of $D_{\lambda}$ filled with 0.    
\end{enumerate}

\end{lem}

\begin{proof}
(a).   Observe that we can obtain $D_{\lambda}$ from $D_{\lambda}'$ by a sequence of steps replacing a single box filled with 1 by a box filled with 0 such that at each step we have a \Le-diagram.  For instance, starting with the first row of $\lambda$ we can move from left to right across the row and replace one box at a time when necessary.  Next, we can repeat the same procedure with the second row, then the third row, and continue in this way until we obtain $D_{\lambda}$.  In this manner we introduce 0's in the \Le-diagram such that at each step the resulting diagram still satisfies the \Le-property.  Let $D_{\mu}$ be a \Le-diagram appearing at some point in this process, and let $D_{\mu}'$ be the next diagram in the sequence.  It suffices to show that the statement is true for the corresponding plabic graphs $G_{\mu}$ and $G_{\mu}'$.   

We proceed by induction on the number of 0's. If $G_{\mu}'$ contains a single 0 in a box $b_{i,j}$, then it must appear in the first row or in the first column of $\mu$.  Here $G_{\mu}=G_{\lambda}$.  We illustrate the three cases in Figure~\ref{fig:basecase}.  We note that the boxes located to the right or below $b_{i,j}$ might not be present in $\mu$, so by including them in the figure we depict the most general situation.  In the first column we draw the local configuration of $G_{\mu}$, and in the second column we show the plabic graphs obtained by removing the edges between faces in $\mathcal{S}_{i,j}$ of $G_{\mu}'$.  We label the resulting face $\mathcal{F}_{i,j}$.  The graphs in the third column are obtained by applying the homotopy relations and removing/adding vertices of degree two.  We can see that in this way we produce the plabic graph corresponding to $G_{\mu}'$.  Note that if some of the boxes to the right or below $b_{i,j}$ are not present in $\mu$ then, by truncating the diagrams in Figure~\ref{fig:basecase} at the appropriate places, we can still observe the same transformation from $G_{\mu}$ to $G_{\mu'}$ and obtain the same conclusion.  This shows the base case.  

\begin{figure}[htb]
  \centering

\begingroup%
  \makeatletter%
  \def\svgwidth{450pt}
  \providecommand\color[2][]{%
    \errmessage{(Inkscape) Color is used for the text in Inkscape, but the package 'color.sty' is not loaded}%
    \renewcommand\color[2][]{}%
  }%
  \providecommand\transparent[1]{%
    \errmessage{(Inkscape) Transparency is used (non-zero) for the text in Inkscape, but the package 'transparent.sty' is not loaded}%
    \renewcommand\transparent[1]{}%
  }%
  \providecommand\rotatebox[2]{#2}%
  \ifx\svgwidth\undefined%
    \setlength{\unitlength}{990.17847124bp}%
    \ifx\svgscale\undefined%
      \relax%
    \else%
      \setlength{\unitlength}{\unitlength * \real{\svgscale}}%
    \fi%
  \else%
    \setlength{\unitlength}{\svgwidth}%
  \fi%
  \global\let\svgwidth\undefined%
  \global\let\svgscale\undefined%
  \makeatother%
  \begin{picture}(1,0.45262927)%
    \put(0,0){\includegraphics[width=\unitlength]{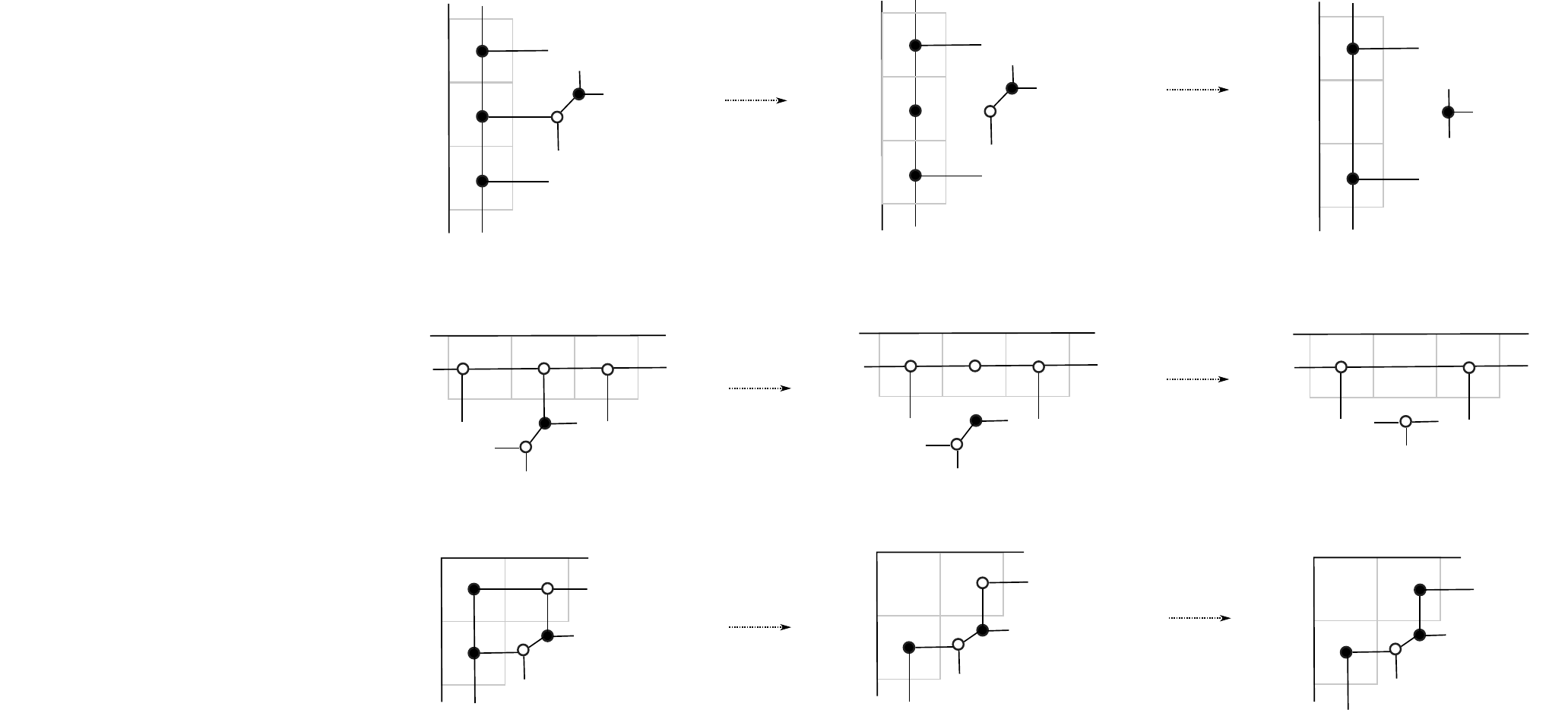}}%
    \put(0.30677121,0.09471848){\color[rgb]{0,0,0}\makebox(0,0)[lt]{\begin{minipage}{0.05794578\unitlength}\raggedright $\scriptstyle{F_0}$\end{minipage}}}%
    \put(0.30717929,0.06656172){\color[rgb]{0,0,0}\makebox(0,0)[lt]{\begin{minipage}{0.05794578\unitlength}\raggedright $\scriptstyle{F_{1,1}}$\end{minipage}}}%
    \put(0.00544988,0.06447718){\color[rgb]{0,0,0}\makebox(0,0)[lt]{\begin{minipage}{0.10446564\unitlength}\raggedright $\scriptstyle{i=1,\,j=1}$\end{minipage}}}%
    \put(0.00130935,0.22384473){\color[rgb]{0,0,0}\makebox(0,0)[lt]{\begin{minipage}{0.10446564\unitlength}\raggedright $\scriptstyle{i=1,\,j\not=1}$\end{minipage}}}%
    \put(0.00132756,0.19458369){\color[rgb]{0,0,0}\makebox(0,0)[lt]{\begin{minipage}{0.11426512\unitlength}\raggedright $\scriptstyle{\mathcal{S}_{1,j}=\{F_{1,j-1},\, F_{1,j}\}}$\end{minipage}}}%
    \put(0.00047103,0.40687271){\color[rgb]{0,0,0}\makebox(0,0)[lt]{\begin{minipage}{0.10446564\unitlength}\raggedright $\scriptstyle{i\not=1,\,j=1}$\end{minipage}}}%
    \put(-0.0000036,0.37592043){\color[rgb]{0,0,0}\makebox(0,0)[lt]{\begin{minipage}{0.11426512\unitlength}\raggedright $\scriptstyle{\mathcal{S}_{i,1}=\{F_{i-1,1},\, F_{i,1}\}}$\end{minipage}}}%
    \put(0.00519027,0.03889605){\color[rgb]{0,0,0}\makebox(0,0)[lt]{\begin{minipage}{0.11426512\unitlength}\raggedright $\scriptstyle{\mathcal{S}_{1,1}=\{F_{0},\, F_{1,1}\}}$\end{minipage}}}%
    \put(0.58093343,0.07431503){\color[rgb]{0,0,0}\makebox(0,0)[lt]{\begin{minipage}{0.05794578\unitlength}\raggedright $\scriptstyle{\mathcal{F}_{i,j}}$\end{minipage}}}%
    \put(0.28515008,0.40407165){\color[rgb]{0,0,0}\makebox(0,0)[lt]{\begin{minipage}{0.05794578\unitlength}\raggedright $\scriptstyle{F_0}$\end{minipage}}}%
    \put(0.31208263,0.37142615){\color[rgb]{0,0,0}\makebox(0,0)[lt]{\begin{minipage}{0.06569909\unitlength}\raggedright $\scriptstyle{F_{i,1}}$\end{minipage}}}%
    \put(0.31248418,0.41735269){\color[rgb]{0,0,0}\makebox(0,0)[lt]{\begin{minipage}{0.05794578\unitlength}\raggedright $\scriptstyle{F_{i\text{-}1,1}}$\end{minipage}}}%
    \put(0.56134613,0.4055187){\color[rgb]{0,0,0}\makebox(0,0)[lt]{\begin{minipage}{0.05794578\unitlength}\raggedright $\scriptstyle{F_0}$\end{minipage}}}%
    \put(0.59113516,0.38470719){\color[rgb]{0,0,0}\makebox(0,0)[lt]{\begin{minipage}{0.05794578\unitlength}\raggedright $\scriptstyle{\mathcal{F}_{i,j}}$\end{minipage}}}%
    \put(0.84007915,0.40129798){\color[rgb]{0,0,0}\makebox(0,0)[lt]{\begin{minipage}{0.05794578\unitlength}\raggedright $\scriptstyle{F_0}$\end{minipage}}}%
    \put(0.87150045,0.38334295){\color[rgb]{0,0,0}\makebox(0,0)[lt]{\begin{minipage}{0.05794578\unitlength}\raggedright $\scriptstyle{\mathcal{F}_{i,j}}$\end{minipage}}}%
    \put(0.85068893,0.08029604){\color[rgb]{0,0,0}\makebox(0,0)[lt]{\begin{minipage}{0.05794578\unitlength}\raggedright $\scriptstyle{\mathcal{F}_{i,j}}$\end{minipage}}}%
    \put(0.30800195,0.23497361){\color[rgb]{0,0,0}\makebox(0,0)[lt]{\begin{minipage}{0.14200797\unitlength}\raggedright $\scriptstyle{F_0}$\end{minipage}}}%
    \put(0.29569213,0.20205365){\color[rgb]{0,0,0}\makebox(0,0)[lt]{\begin{minipage}{0.04637331\unitlength}\raggedright $\scriptstyle{F_{1,j\text{-}1}}$\end{minipage}}}%
    \put(0.34651517,0.21221969){\color[rgb]{0,0,0}\makebox(0,0)[lt]{\begin{minipage}{0.14200797\unitlength}\raggedright $\scriptstyle{F_{1,j}}$\end{minipage}}}%
    \put(0.58930364,0.23564864){\color[rgb]{0,0,0}\makebox(0,0)[lt]{\begin{minipage}{0.14200797\unitlength}\raggedright $\scriptstyle{F_0}$\end{minipage}}}%
    \put(0.60937947,0.20871609){\color[rgb]{0,0,0}\makebox(0,0)[lt]{\begin{minipage}{0.14200797\unitlength}\raggedright $\scriptstyle{\mathcal{F}_{1,j}}$\end{minipage}}}%
    \put(0.85742688,0.23550861){\color[rgb]{0,0,0}\makebox(0,0)[lt]{\begin{minipage}{0.14200797\unitlength}\raggedright $\scriptstyle{F_0}$\end{minipage}}}%
    \put(0.89089479,0.21020834){\color[rgb]{0,0,0}\makebox(0,0)[lt]{\begin{minipage}{0.14200797\unitlength}\raggedright $\scriptstyle{\mathcal{F}_{i,j}}$\end{minipage}}}%
  \end{picture}%
\endgroup%

  \caption{Plabic graphs obtained by changing the filling of a box $b_{i,j}$ from 1 to 0 in the \Le-diagram while the rest of the boxes are filled with 1's.}
   \label{fig:basecase}
\end{figure}
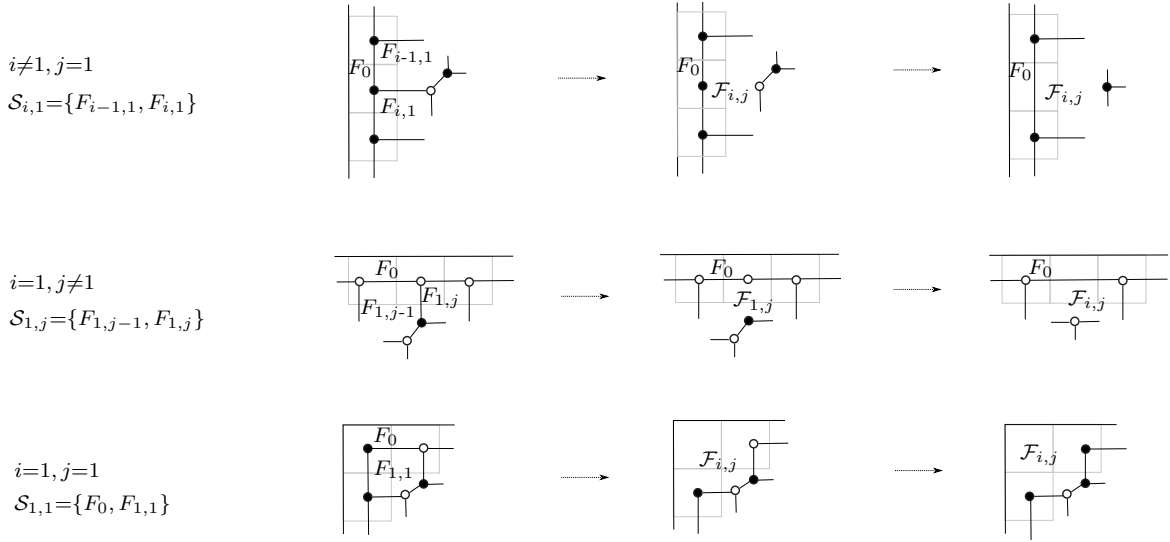

Suppose that the statement holds whenever $G_{\mu}$ contains $m\geq 1$ boxes filled with 0's.  We want to show that the lemma is also true for $G_{\mu}'$ which is obtained from $G_{\mu}$ by adding one more 0 in a box $b_{i,j}$.   Recall that by construction of $G_{\mu}'$ as described in the first paragraph there are no 0's in the boxes below and to the right of $b_{i,j}$.  This allows us to better depict the local configuration of the corresponding face $F_{i,j}$ in $G_{\mu}$ (see Figure~\ref{fig:lemma2}).  By induction the graph $G_{\mu}$ is obtained from $G_{\lambda}$ by removing edges between the appropriate faces.   Given a face $F_{x}$ in $G_{\lambda}$ let $\mathcal{F}_{x}$ denote the face in $G_{\mu}$ resulting from removing edges between $F_{x}$ and other faces as described in the induction hypothesis.  Observe that in $G_{\mu}$ we have $F_{i',j'}=\mathcal{F}_{i',j'}$ for all $i'\geq i$, $j'\geq j$.   However, we remark that combining the faces in $\mathcal{S}_x$ does not necessarily result in $\mathcal{F}_x$, because in general $\mathcal{S}_x$ forms only a subset of faces that make up $\mathcal{F}_x$.

In the first column in Figure~\ref{fig:lemma2} we show various possibilities for the local configuration of $F_{i,j}$ in $G_{\mu}$.  The dotted and dashed lines drawn inside the same rectangle represent two different local configurations that might occur in a graph.  The solid black lines depict edges that must be present in every such plabic graph regardless of choosing dotted or dashed edges to complete the graph.  We show these possibilities simultaneously on the same diagram to avoid considering too many cases that can be handled by the exact same argument.   Also, note that some of the boxes to the right and below $b_{i,j}$ might not be present in $D_{\mu}$, but here we depict the most general situation.  In the second column we draw graphs obtained from $G_{\mu}$ after eliminating edge(s) between the designated faces.  We label the resulting face $\mathcal{F'}_{i,j}$.  The graphs in the third column are obtained from those in the second column by removing/adding vertices of degree 2 and applying homotopy relations.  The grey vertex in the bottom right diagram represents a vertex that could be either white or black depending on the choice of the dotted or dashed edges.  We claim that the graphs in the third column are precisely $G_{\mu}'$.  We discuss the different cases in more detail below.

In case 1 we assume that all boxes to the left of $b_{i,j}$ are filled with 0's but there is some box above $b_{i,j}$ that is filled with 1.  In particular, we see that $i\not=1$.  We want to show that $G_{\mu}'$ is obtained from $G_{\mu}$ by removing the edge between $F_{i,j}$ and $\mathcal{F}_{i-1,j}$, because by definition we have $\mathcal{S}_{i,j} = \mathcal{S}_{i-1,j}\cup \{F_{i,j}\}$ in $G_{\mu}'$.   Observe that the associated $\Gamma$-graphs $\mathcal{G}_{\mu}, \mathcal{G}_{\mu}'$ differ by a single internal vertex $x_{i,j}$ and its adjacent edges.  Recall that by assumption there exists an internal vertex in $\mathcal{G}_{\mu}$ above $x_{i,j}$.   Hence, by Remark~\ref{internalvertex}, except for $x_{i,j}$ and the vertex $x_{i, j+1}$ located directly to the right of $x_{i,j}$, the local transformation given in Figure~\ref{fig:transformation} of the remaining internal vertices is the same for both $\mathcal{G}_{\mu}$ and $\mathcal{G}_{\mu}'$.   Therefore, it suffices to check that $G_{\mu}'$ coincides with the plabic graph obtained in the third column in Figure~\ref{fig:lemma2} in the neighborhood of boxes $b_{i,j}$ and $b_{i, j+1}$.  However, this is precisely what the figure shows.  Note that if some of the boxes to the right or below $b_{i,j}$ are not present in $\mu$ then, by truncating the diagrams in Figure~\ref{fig:lemma2} at the appropriate places, we can still observe the same transformation from $G_{\mu}$ to $G_{\mu'}$ and obtain the same conclusion.

In case 2 we assume that all boxes above $b_{i,j}$ are filled with 0's but there is some box to the left of $b_{i,j}$ that is filled with 1.  In particular, we see that $j\not=1$.  We want to show that $G_{\mu}'$ is obtained from $G_{\mu}$ by removing the edge between $F_{i,j}$ and $\mathcal{F}_{i,j-1}$, because by definition we have $\mathcal{S}_{i,j} = \mathcal{S}_{i,j-1}\cup \{F_{i,j}\}$ in $G_{\mu}'$.  By the same argument as in case 1 it suffices to check that the two plabic graphs coincide in the neighborhood of boxes $b_{i,j}$ and $b_{i+1,j}$.  However, we can see this is true from Figure~\ref{fig:lemma2}.

In case 3 we assume that all boxes above and to the left of $b_{i,j}$ are filled with 0's.   First observe that in this case we have $\mathcal{F}_{i,j-1}=\mathcal{F}_{i-1,j}=\mathcal{F}_{i-1,j-1}$, because by definition  $F_{i-1,j-1}$ belongs to both $\mathcal{S}_{i-1,j}$ and $\mathcal{S}_{i,j-1}$.   Also, in $G_{\mu}'$ the faces $F_{i-1,j}, F_{i,j-1}\in \mathcal{S}_{i,j}$, and these are the only faces in $\mathcal{S}_{i,j}$ such that the corresponding $\mathcal{F}_{i-1,j}, \mathcal{F}_{i,j-1}$ share an edge with $F_{i,j}$.  Hence, in accordance with the lemma $\mathcal{F}'_{i,j}$ should be obtained by removing edges between $F_{i,j}$ and $\mathcal{F}_{i,j-1}$.   Using the same argument as in the two cases discussed above, it suffices to check that $G_{\mu}'$ and the bottom right graph in Figure~\ref{fig:lemma2} coincide in the neighborhood of boxes $b_{i,j}, b_{i+1, j}, b_{i,j+1}$.  However, we can again see this is true from Figure~\ref{fig:lemma2}.  

By induction this completes the proof.  We remark that the order of removing edges between the faces in a given $\mathcal{S}_{i,j}$ does not make a difference. 

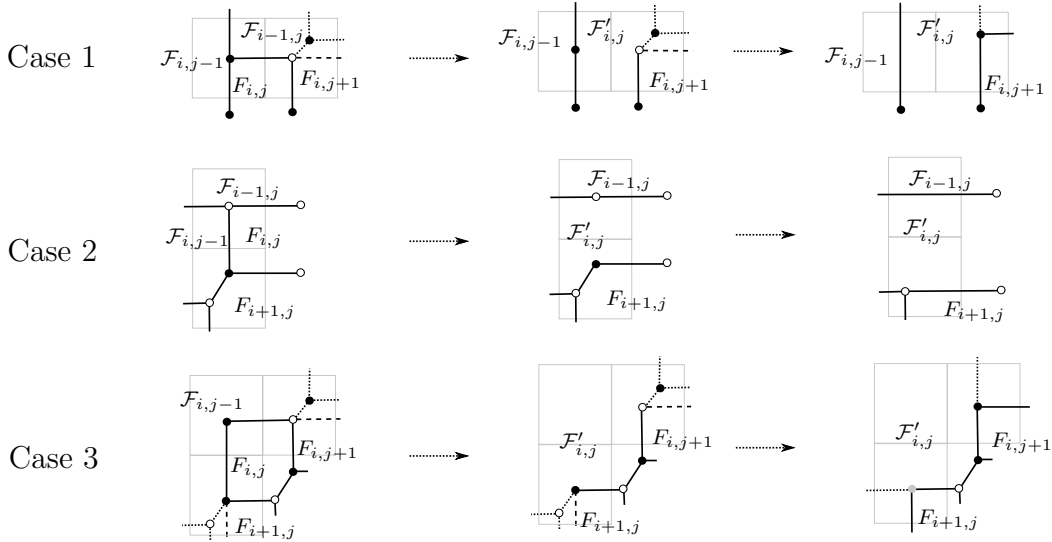
\begin{figure}[htb]
  \centering

\begingroup%
  \makeatletter%
  \def\svgwidth{400pt}
  \providecommand\color[2][]{%
    \errmessage{(Inkscape) Color is used for the text in Inkscape, but the package 'color.sty' is not loaded}%
    \renewcommand\color[2][]{}%
  }%
  \providecommand\transparent[1]{%
    \errmessage{(Inkscape) Transparency is used (non-zero) for the text in Inkscape, but the package 'transparent.sty' is not loaded}%
    \renewcommand\transparent[1]{}%
  }%
  \providecommand\rotatebox[2]{#2}%
  \ifx\svgwidth\undefined%
    \setlength{\unitlength}{1016.1137581bp}%
    \ifx\svgscale\undefined%
      \relax%
    \else%
      \setlength{\unitlength}{\unitlength * \real{\svgscale}}%
    \fi%
  \else%
    \setlength{\unitlength}{\svgwidth}%
  \fi%
  \global\let\svgwidth\undefined%
  \global\let\svgscale\undefined%
  \makeatother%
  \begin{picture}(1,0.52304675)%
    \put(0,0){\includegraphics[width=\unitlength]{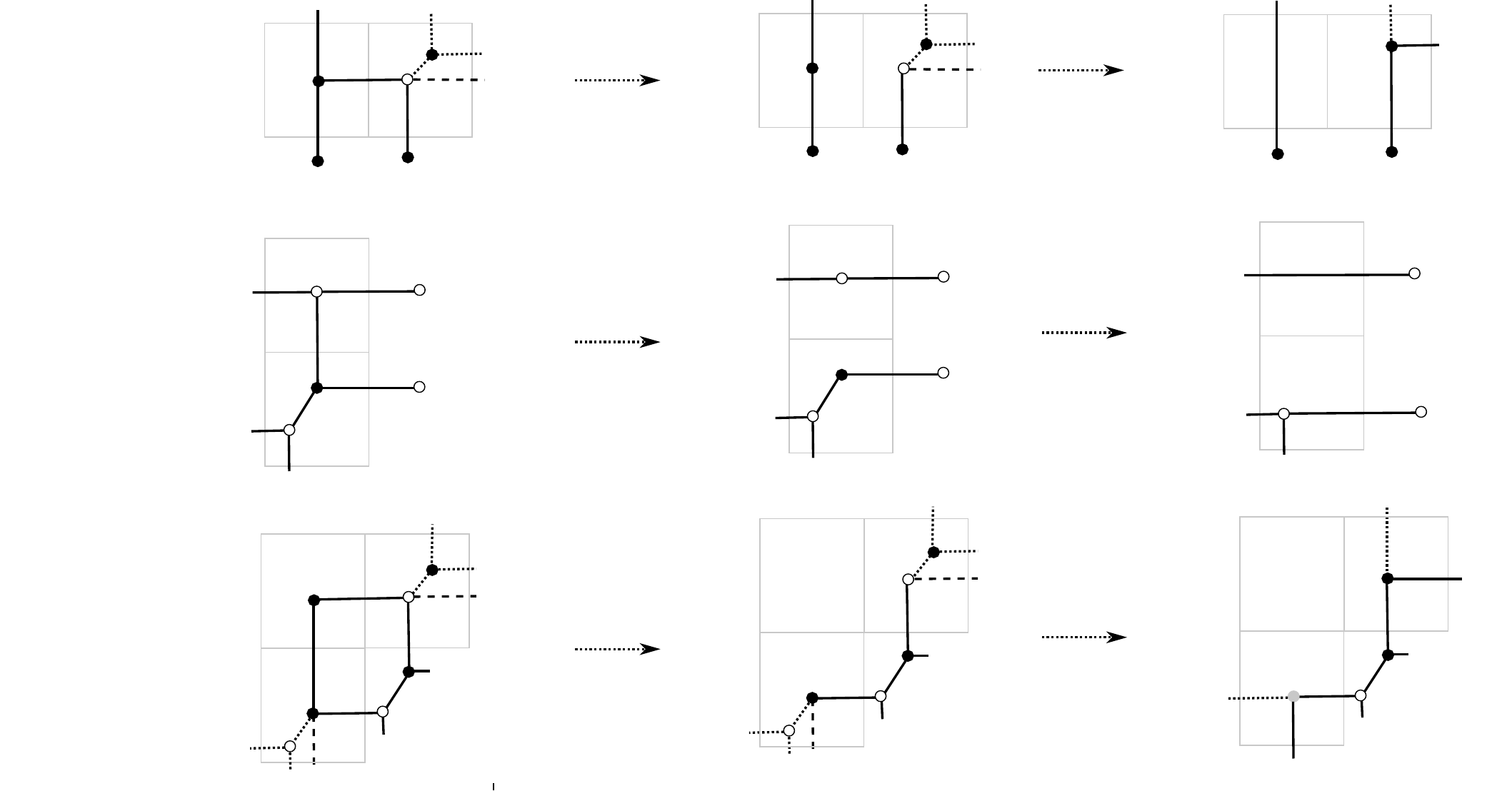}}%
    \put(0.22512541,0.30943968){\color[rgb]{0,0,0}\makebox(0,0)[lt]{\begin{minipage}{0.1130461\unitlength}\raggedright $\scriptstyle{F_{i,j}}$\end{minipage}}}%
    \put(0.14988158,0.30879691){\color[rgb]{0,0,0}\makebox(0,0)[lt]{\begin{minipage}{0.1130461\unitlength}\raggedright $\scriptstyle{\mathcal{F}_{i,j-1}}$\end{minipage}}}%
    \put(0.1977573,0.35573956){\color[rgb]{0,0,0}\makebox(0,0)[lt]{\begin{minipage}{0.1130461\unitlength}\raggedright $\scriptstyle{\mathcal{F}_{i-1,j}}$\end{minipage}}}%
    \put(0.27612668,0.45852362){\color[rgb]{0,0,0}\makebox(0,0)[lt]{\begin{minipage}{0.1130461\unitlength}\raggedright $\scriptstyle{F_{i,j+1}}$\end{minipage}}}%
    \put(0.21417149,0.24305778){\color[rgb]{0,0,0}\makebox(0,0)[lt]{\begin{minipage}{0.1130461\unitlength}\raggedright $\scriptstyle{F_{i+1,j}}$\end{minipage}}}%
    \put(-0.00030695,0.48177935){\color[rgb]{0,0,0}\makebox(0,0)[lt]{\begin{minipage}{0.14833749\unitlength}\raggedright Case 1\end{minipage}}}%
    \put(0.00018303,0.29829077){\color[rgb]{0,0,0}\makebox(0,0)[lt]{\begin{minipage}{0.14981892\unitlength}\raggedright Case 2\end{minipage}}}%
    \put(0.00146833,0.10262036){\color[rgb]{0,0,0}\makebox(0,0)[lt]{\begin{minipage}{0.14257826\unitlength}\raggedright Case 3\end{minipage}}}%
    \put(0.21230268,0.45147214){\color[rgb]{0,0,0}\makebox(0,0)[lt]{\begin{minipage}{0.1130461\unitlength}\raggedright $\scriptstyle{F_{i,j}}$\end{minipage}}}%
    \put(0.21063253,0.09239119){\color[rgb]{0,0,0}\makebox(0,0)[lt]{\begin{minipage}{0.1130461\unitlength}\raggedright $\scriptstyle{F_{i,j}}$\end{minipage}}}%
    \put(0.2207184,0.50380331){\color[rgb]{0,0,0}\makebox(0,0)[lt]{\begin{minipage}{0.1130461\unitlength}\raggedright $\scriptstyle{\mathcal{F}_{i-1,j}}$\end{minipage}}}%
    \put(0.1450052,0.47763772){\color[rgb]{0,0,0}\makebox(0,0)[lt]{\begin{minipage}{0.1130461\unitlength}\raggedright $\scriptstyle{\mathcal{F}_{i,j-1}}$\end{minipage}}}%
    \put(0.16337678,0.15474323){\color[rgb]{0,0,0}\makebox(0,0)[lt]{\begin{minipage}{0.1130461\unitlength}\raggedright $\scriptstyle{\mathcal{F}_{i,j-1}}$\end{minipage}}}%
    \put(0.2149336,0.03449287){\color[rgb]{0,0,0}\makebox(0,0)[lt]{\begin{minipage}{0.1130461\unitlength}\raggedright $\scriptstyle{F_{i+1,j}}$\end{minipage}}}%
    \put(0.27338864,0.10853592){\color[rgb]{0,0,0}\makebox(0,0)[lt]{\begin{minipage}{0.1130461\unitlength}\raggedright $\scriptstyle{F_{i,j+1}}$\end{minipage}}}%
    \put(0.54862332,0.36406793){\color[rgb]{0,0,0}\makebox(0,0)[lt]{\begin{minipage}{0.1130461\unitlength}\raggedright $\scriptstyle{\mathcal{F}_{i-1,j}}$\end{minipage}}}%
    \put(0.52907331,0.31452033){\color[rgb]{0,0,0}\makebox(0,0)[lt]{\begin{minipage}{0.1130461\unitlength}\raggedright $\scriptstyle{\mathcal{F}'_{i,j}}$\end{minipage}}}%
    \put(0.56622054,0.24994143){\color[rgb]{0,0,0}\makebox(0,0)[lt]{\begin{minipage}{0.1130461\unitlength}\raggedright $\scriptstyle{F_{i+1,j}}$\end{minipage}}}%
    \put(0.46444071,0.49687575){\color[rgb]{0,0,0}\makebox(0,0)[lt]{\begin{minipage}{0.1130461\unitlength}\raggedright $\scriptstyle{\mathcal{F}_{i,j-1}}$\end{minipage}}}%
    \put(0.54838106,0.50770032){\color[rgb]{0,0,0}\makebox(0,0)[lt]{\begin{minipage}{0.1130461\unitlength}\raggedright $\scriptstyle{\mathcal{F}'_{i,j}}$\end{minipage}}}%
    \put(0.52221547,0.11632992){\color[rgb]{0,0,0}\makebox(0,0)[lt]{\begin{minipage}{0.1130461\unitlength}\raggedright $\scriptstyle{\mathcal{F}'_{i,j}}$\end{minipage}}}%
    \put(0.60630399,0.45648256){\color[rgb]{0,0,0}\makebox(0,0)[lt]{\begin{minipage}{0.1130461\unitlength}\raggedright $\scriptstyle{F_{i,j+1}}$\end{minipage}}}%
    \put(0.54116838,0.03950329){\color[rgb]{0,0,0}\makebox(0,0)[lt]{\begin{minipage}{0.1130461\unitlength}\raggedright $\scriptstyle{F_{i+1,j}}$\end{minipage}}}%
    \put(0.60964427,0.11855678){\color[rgb]{0,0,0}\makebox(0,0)[lt]{\begin{minipage}{0.1130461\unitlength}\raggedright $\scriptstyle{F_{i,j+1}}$\end{minipage}}}%
    \put(0.85258953,0.36406793){\color[rgb]{0,0,0}\makebox(0,0)[lt]{\begin{minipage}{0.1130461\unitlength}\raggedright $\scriptstyle{\mathcal{F}_{i-1,j}}$\end{minipage}}}%
    \put(0.84733681,0.31897405){\color[rgb]{0,0,0}\makebox(0,0)[lt]{\begin{minipage}{0.1130461\unitlength}\raggedright $\scriptstyle{\mathcal{F}'_{i,j}}$\end{minipage}}}%
    \put(0.88521808,0.23992057){\color[rgb]{0,0,0}\makebox(0,0)[lt]{\begin{minipage}{0.1130461\unitlength}\raggedright $\scriptstyle{F_{i+1,j}}$\end{minipage}}}%
    \put(0.77320314,0.48009518){\color[rgb]{0,0,0}\makebox(0,0)[lt]{\begin{minipage}{0.1130461\unitlength}\raggedright $\scriptstyle{\mathcal{F}_{i,j-1}}$\end{minipage}}}%
    \put(0.86014128,0.50770031){\color[rgb]{0,0,0}\makebox(0,0)[lt]{\begin{minipage}{0.1130461\unitlength}\raggedright $\scriptstyle{\mathcal{F}'_{i,j}}$\end{minipage}}}%
    \put(0.84009959,0.12523735){\color[rgb]{0,0,0}\makebox(0,0)[lt]{\begin{minipage}{0.1130461\unitlength}\raggedright $\scriptstyle{\mathcal{F}'_{i,j}}$\end{minipage}}}%
    \put(0.92363142,0.45314228){\color[rgb]{0,0,0}\makebox(0,0)[lt]{\begin{minipage}{0.1130461\unitlength}\raggedright $\scriptstyle{F_{i,j+1}}$\end{minipage}}}%
    \put(0.86072263,0.04284359){\color[rgb]{0,0,0}\makebox(0,0)[lt]{\begin{minipage}{0.1130461\unitlength}\raggedright $\scriptstyle{F_{i+1,j}}$\end{minipage}}}%
    \put(0.93031193,0.1157732){\color[rgb]{0,0,0}\makebox(0,0)[lt]{\begin{minipage}{0.1130461\unitlength}\raggedright $\scriptstyle{F_{i,j+1}}$\end{minipage}}}%
  \end{picture}%
\endgroup%

  \caption{Plabic graphs obtained by changing the filling of a box $b_{i,j}$ from 1 to 0 in the \Le-diagram.}
   \label{fig:lemma2}
\end{figure}

Part (b) follows from part (a) and Lemma~\ref{lemma0}.
\end{proof}

Next we present a procedure for constructing a quiver associated to an arbitrary \Le-diagram.  

\begin{thm}\label{quiverconstruction}
Let $D_{\mu}$ be a \Le-diagram of shape $\mu$.  Let $D_{\lambda}$ be a \Le-diagram filled only with 1's of rectangular shape $\lambda$ such that $\mu\subset \lambda$.  Let $Q_{\lambda}$ and $Q_{\mu}$ be the corresponding quivers.  Then $Q_{\mu}$ can be obtained from $Q_{\lambda}$ by applying the following steps. 

\medskip

\begin{enumerate}[1.]

\item Remove the vertices $v_{i,j}$ corresponding to boxes $b_{i,j}$ in $\lambda\setminus\mu$ and the arrows between the resulting frozen vertices to obtain a new quiver $Q_{\mu}'$. 

\item Merge the vertices of $Q_{\mu}'$ belonging to $S_{i,j}$ for every box $b_{i,j}$ of $D_{\mu}$ filled with 0. 

\end{enumerate}

\end{thm}

 \begin{proof}
 The theorem follows directly from Lemma~\ref{lemma1}(b) and Lemma~\ref{lemma2}(b).
 \end{proof}
 
 \begin{exmp}
If $D_{\mu}$ is given in Figure~\ref{fig:Le}, then let $D_{\lambda}$ be the diagram of rectangular shape $4\times 5$ filled only with 1's.  The corresponding quiver $Q_{\lambda}$ is given in Figure~\ref{fig:grass}, where the frozen vertices are colored blue.  With the notation of Theorem~\ref{quiverconstruction} the quivers $Q_{\mu}'$, $Q_{\mu}$ are given in Figure~\ref{fig:exquiver}.  The sets $S_{i,j}$ are given in Example~\ref{S}, hence $Q_{\mu}$ is obtained from $Q_{\mu}'$ by merging the following sets of vertices.  Observe that the only nonfrozen vertex in $Q_{\mu}$ is $v_{1,2}$ obtained by merging vertices in $S_{1,3}$.  

$$\xymatrix@R=0pt{S_{1,1}=\{v_0, v_{1,1} \} && S_{2,3}\cup S_{3,2}= \{v_{2,2}, v_{2,3}, v_{3,2}\} \\
S_{1,3}=\{v_{1,2}, v_{1,3}\} && S_{3,1}\cup S_{4,1} =\{v_{2,1}, v_{3,1}, v_{4,1}\} \\
S_{1,5}=\{ v_{1,4}, v_{1,5}\} &&  }$$

Moreover, using Lemma~\ref{lemma1}(a) and Lemma~\ref{lemma2}(a) we can also obtain the corresponding plabic graph $G_{\mu}$ from the corresponding $G_{\lambda}$ given in Figure~\ref{fig:grass}.    The graph $G_{\mu}$ is given in Figure~\ref{fig:exquiver}, and we can see that $Q_{\mu}$ is the quiver associated to this plabic graph.  
 \medskip

\begin{figure}[htb]

\hspace{12cm}{\begingroup%
  \makeatletter%
  \def\svgwidth{120pt}
  \providecommand\color[2][]{%
    \errmessage{(Inkscape) Color is used for the text in Inkscape, but the package 'color.sty' is not loaded}%
    \renewcommand\color[2][]{}%
  }%
  \providecommand\transparent[1]{%
    \errmessage{(Inkscape) Transparency is used (non-zero) for the text in Inkscape, but the package 'transparent.sty' is not loaded}%
    \renewcommand\transparent[1]{}%
  }%
  \providecommand\rotatebox[2]{#2}%
  \ifx\svgwidth\undefined%
    \setlength{\unitlength}{201.625bp}%
    \ifx\svgscale\undefined%
      \relax%
    \else%
      \setlength{\unitlength}{\unitlength * \real{\svgscale}}%
    \fi%
  \else%
    \setlength{\unitlength}{\svgwidth}%
  \fi%
  \global\let\svgwidth\undefined%
  \global\let\svgscale\undefined%
  \makeatother%
  \begin{picture}(1,0.79764414)%
    \put(0,0){\includegraphics[width=\unitlength]{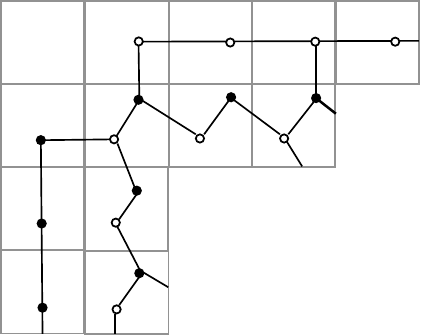}}%
  \end{picture}%
\endgroup%
}

 \vspace{-4cm}\hspace{-6cm}$\xymatrix@C=10pt@R=10pt{{\color{blue} v_0} &&&&&& \\
v_{1,1}\ar[u] \ar[dr]&v_{1,2} \ar[l]\ar[dr]&v_{1,3}\ar[l]\ar[dr]&{\color{blue} v_{1,4}}\ar[l]&{\color{blue} v_{1,5}}&&& {\color{blue} v_0}&v_{1,2}\ar[l]\ar[dr]&{\color{blue} v_{1,4}}\ar[l]\\
  v_{2,1}\ar[u]\ar[dr] & {\color{blue} v_{2,2}}\ar[l]\ar[u] &{\color{blue}  v_{2,3}}\ar[u] & {\color{blue} v_{2,4}} & &&&{\color{blue} v_{2,1}}& {\color{blue} v_{2,2}}\ar[u]&{\color{blue} v_{2,4}}\\
   v_{3,1}\ar[u] & {\color{blue} v_{3,2}}\ar[l]& & & & &&& {\color{blue}  v_{4,2}}\\
   {\color{blue} v_{4,1}}\ar[u] &{\color{blue} v_{4,2}} & & &  }$

  \caption{The quiver $Q_{\mu}'$ (left), the corresponding $Q_{\mu}$ (middle), and the associated plabic graph $G_{\mu}$ (right).}
   \label{fig:exquiver}
\end{figure}

 \end{exmp}
 
\section{Existence of a green-to-red sequence}
In this section we prove the existence of a green-to-red sequence for quivers arising from \Le-diagrams. This enables us to deduce that quivers coming from reduced plabic graphs also admit such sequences.  As mentioned earlier, we will not construct the sequence directly. Rather, its existence follows from the following theorem:

\begin{thm}\label{mainthm}
Let $Q$ be a quiver arising from a \Le-diagram $D$ of shape $\lambda$, where $\lambda$ fits in a $k\times n$ rectangle. Let $D''$ be the \Le-diagram whose shape is a $k\times n$ rectangle containing all 1's, and let $Q''$ be the corresponding quiver. Then there is a sequence of vertex deletions and mutations taking $Q''$ to $Q$.
\end{thm}

\begin{cor}\label{maincor}
Quivers associated to reduced plabic graphs admit a green-to-red sequence.
\end{cor}

\begin{proof}[Proof of Corollary \ref{maincor}] Given a quiver associated to a reduced plabic graph, there exists a quiver $Q$ in its mutation class that arises from a \Le-diagram. The corresponding quiver $Q''$, as defined in the statement of Theorem \ref{mainthm}, admits a green-to-red sequence by \cite[Theorem 11.17]{MS}.  The rest follows directly from Theorem \ref{mainthm}: by Theorem \ref{gtrdelete}, deleting a vertex preserves the existence of a green-to-red sequence, and by Theorem \ref{gtrmutate}, mutating does as well.
\end{proof}

\begin{proof}[Proof of Theorem \ref{mainthm}]
First, let $Q'$ be the quiver corresponding to the \Le-diagram $D'$ of shape $\lambda$ filled with all 1's. By Lemma \ref{lemma1}, we may get to $Q'$ from $Q''$ by deleting the vertices of $Q''$ corresponding to cells of the rectangle that don't appear in $\lambda$. So it is sufficient to find a sequence of deletions and mutations taking $Q'$ to $Q$.

As in Lemma \ref{lemma2}, we can form $D$ from $D'$ by changing 1's to 0's one at a time, moving from left to right and top to bottom. We know from Lemma \ref{lemma2} that this operation corresponds on the quiver to merging vertices. We claim that the result of merging any pair of vertices that arises during this process can be obtained by a sequence of quiver mutations followed by a deletion.

Suppose we are replacing a 1 in box $b_{i,j}$ of $D'$ with a 0. In all cases, since we have not yet removed any 1's from the portion of the diagram southeast of $b_{i,j}$, that portion of $D'$ (and therefore $Q'$) looks exactly like the corresponding part of Figure \ref{fig:grass}, except possibly with vertices removed from the lower-right corner. (Recall also that we are free to ignore the frozen vertices, since they do not affect the existence of a green-to-red sequence.) We mutate at the vertex $v_{i,j}$, then $v_{i+1,j+1}$, and so on, moving southeast one step at a time, until there is no longer a vertex to the southeast of the last vertex used. It is helpful to move a vertex $v_{a,b}$ one step to the southeast as you mutate it so that it occupies the position of $v_{a+1,b+1}$, as in Figure \ref{fig:exmutation}.

\begin{figure}[htb]

\resizebox{\textwidth}{!}{
\begingroup%
  \makeatletter%
  \providecommand\color[2][]{%
    \errmessage{(Inkscape) Color is used for the text in Inkscape, but the package 'color.sty' is not loaded}%
    \renewcommand\color[2][]{}%
  }%
  \providecommand\transparent[1]{%
    \errmessage{(Inkscape) Transparency is used (non-zero) for the text in Inkscape, but the package 'transparent.sty' is not loaded}%
    \renewcommand\transparent[1]{}%
  }%
  \providecommand\rotatebox[2]{#2}%
  \ifx\svgwidth\undefined%
    \setlength{\unitlength}{1750.94558441bp}%
    \ifx\svgscale\undefined%
      \relax%
    \else%
      \setlength{\unitlength}{\unitlength * \real{\svgscale}}%
    \fi%
  \else%
    \setlength{\unitlength}{\svgwidth}%
  \fi%
  \global\let\svgwidth\undefined%
  \global\let\svgscale\undefined%
  \makeatother%
  \begin{picture}(1,0.17336094)%
    \put(0,0){\includegraphics[width=\unitlength,page=1]{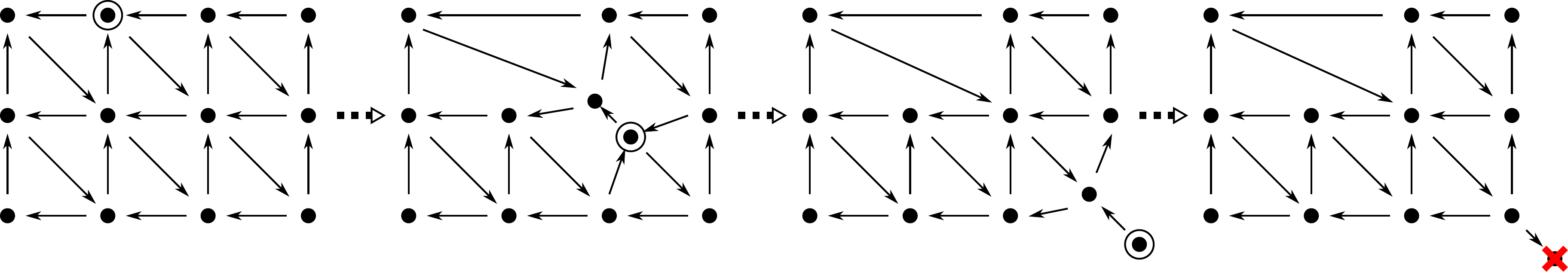}}%
  \end{picture}%
\endgroup%

}
  \caption{An example of the sequence of mutations from the proof of Theorem \ref{mainthm}. At each step we are about to mutate at the circled vertex, which slides to the southeast in the following picture. At the end, the vertex marked with an X is deleted.}
   \label{fig:exmutation}
\end{figure}

After performing these mutations and sliding each mutated vertex one step to the southeast, the resulting quiver has a vertex at every position along the southeast-pointing chain of vertices starting from $v_{i,j}$, except $v_{i,j}$ itself, plus an extra vertex at the end. We delete this extra vertex.

Since placing a 0 in box $b_{i,j}$ of $D'$ leaves a valid \Le-diagram, there are only three possibilities: all boxes to the left of $b_{i,j}$ contain a 0, all boxes above it contain a 0, or both; the corresponding configurations in the plabic graph are the ones enumerated in Figure \ref{fig:lemma2}. In all three cases, there is a unique vertex in the portion of the quiver not southeast of $v_{i,j}$ which is incident to $v_{i,j}$, and that vertex has a single edge from $v_{i,j}$: in the first case, this vertex is above $v_{i,j}$, in the second case it is to the left, and in the third case it is northwest. (In the third case, it might be the frozen vertex $v_0$.) Call this vertex $v'$.

We claim that the result of the procedure, consisting of mutations and a deletion, described above is the same as the result of merging $v_{i,j}$ with $v'$. Since, as we saw in the proof of Lemma~\ref{lemma2}, these are the vertices that are merged when we place a 0 in box $b_{i,j}$, this is enough to prove the theorem.

\begin{figure}[htb]

\begin{subfigure}[t]{0.3\textwidth}
        \centering
        \resizebox{\textwidth}{!}{ \def\svgwidth{120pt} \begingroup%
  \makeatletter%
  \providecommand\color[2][]{%
    \errmessage{(Inkscape) Color is used for the text in Inkscape, but the package 'color.sty' is not loaded}%
    \renewcommand\color[2][]{}%
  }%
  \providecommand\transparent[1]{%
    \errmessage{(Inkscape) Transparency is used (non-zero) for the text in Inkscape, but the package 'transparent.sty' is not loaded}%
    \renewcommand\transparent[1]{}%
  }%
  \providecommand\rotatebox[2]{#2}%
  \ifx\svgwidth\undefined%
    \setlength{\unitlength}{346.4bp}%
    \ifx\svgscale\undefined%
      \relax%
    \else%
      \setlength{\unitlength}{\unitlength * \real{\svgscale}}%
    \fi%
  \else%
    \setlength{\unitlength}{\svgwidth}%
  \fi%
  \global\let\svgwidth\undefined%
  \global\let\svgscale\undefined%
  \makeatother%
  \begin{picture}(1,1.06305806)%
    \put(0,0){\includegraphics[width=\unitlength,page=1]{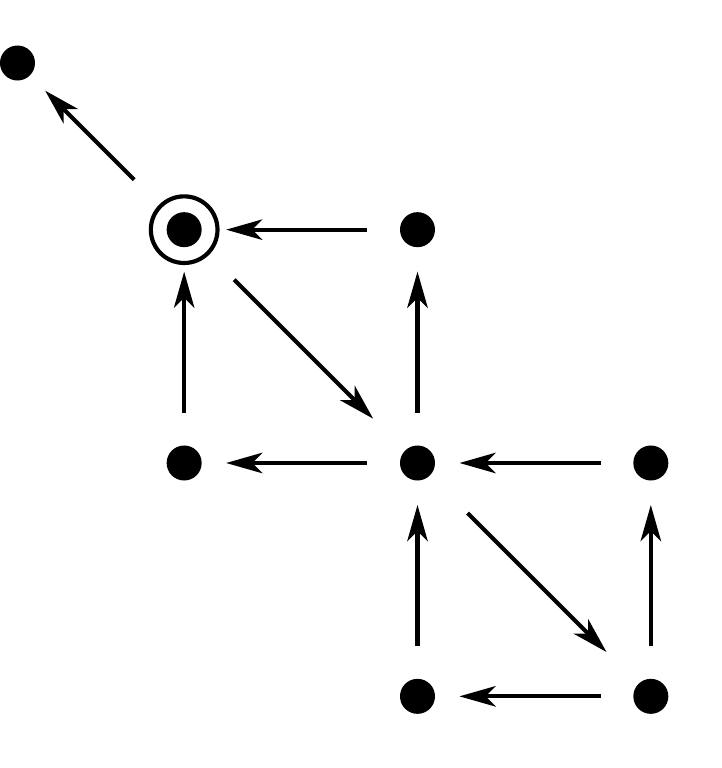}}%
    \put(0.04899372,1.00049491){\color[rgb]{0,0,0}\makebox(0,0)[lb]{\smash{$v'$}}}%
    \put(0,0){\includegraphics[width=\unitlength,page=2]{squares1.pdf}}%
    \put(0.30303529,0.79264273){\color[rgb]{0,0,0}\makebox(0,0)[lb]{\smash{$v_{i,j}$}}}%
  \end{picture}%
\endgroup%
}
        \caption{Before the first mutation\label{fig:chainofsquares1}}
    \end{subfigure}%
\hspace{0.1\textwidth}
\begin{subfigure}[t]{0.3\textwidth}
        \centering
        \resizebox{\textwidth}{!}{\def\svgwidth{120pt}  \begingroup%
  \makeatletter%
  \providecommand\color[2][]{%
    \errmessage{(Inkscape) Color is used for the text in Inkscape, but the package 'color.sty' is not loaded}%
    \renewcommand\color[2][]{}%
  }%
  \providecommand\transparent[1]{%
    \errmessage{(Inkscape) Transparency is used (non-zero) for the text in Inkscape, but the package 'transparent.sty' is not loaded}%
    \renewcommand\transparent[1]{}%
  }%
  \providecommand\rotatebox[2]{#2}%
  \ifx\svgwidth\undefined%
    \setlength{\unitlength}{290.4bp}%
    \ifx\svgscale\undefined%
      \relax%
    \else%
      \setlength{\unitlength}{\unitlength * \real{\svgscale}}%
    \fi%
  \else%
    \setlength{\unitlength}{\svgwidth}%
  \fi%
  \global\let\svgwidth\undefined%
  \global\let\svgscale\undefined%
  \makeatother%
  \begin{picture}(1,1.07521801)%
    \put(0,0){\includegraphics[width=\unitlength,page=1]{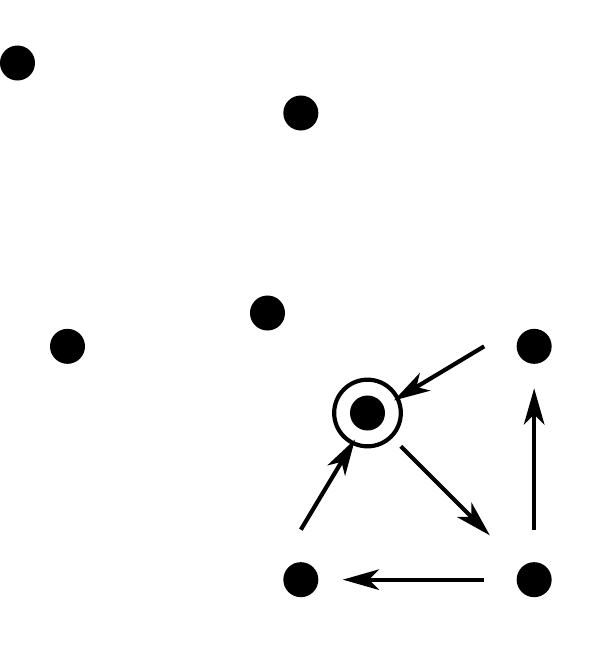}}%
    \put(0.05844154,1.00059035){\color[rgb]{0,0,0}\makebox(0,0)[lb]{\smash{$v'$}}}%
    \put(0,0){\includegraphics[width=\unitlength,page=2]{squares2.pdf}}%
    \put(0.49724518,0.58539949){\color[rgb]{0,0,0}\makebox(0,0)[lb]{\smash{$v_{i,j}$}}}%
  \end{picture}%
\endgroup%
}
        \caption{After mutating at $v_{i,j}$.\label{fig:chainofsquares2}}
    \end{subfigure}%

  \caption{The result of mutating at the upper-left corner of a ``chain of squares,'' as described in the proof of Theorem \ref{mainthm}.}
\end{figure}

To see this, first note that the result of our sequence of mutations depends only on the part of the quiver incident to the vertices in the southeast-pointing chain from $v_{i,j}$. Since the entire portion of the quiver to the southeast of $v_{i,j}$ looks the same as in the all-1's \Le-diagram $Q'$, this part of the quiver looks like a ``chain of squares,'' as in Figure \ref{fig:chainofsquares1}. After mutating at $v_{i,j}$, we see that the edges incident to it now look like those that used to be incident to $v_{i+1,j+1}$, that $v'$ now has all of its old edges plus the edges that used to be incident to $v_{i,j}$ (before canceling 2-cycles), and that the portion of the quiver touching the vertices in the southeast-pointing chain from $v_{i+1,j+1}$ now forms a new chain-of-squares quiver, with $v_{i,j}$ now serving the role of $v'$. We continue in this way until we arrive at the end of the diagram.

\begin{figure}[htb]

\begin{subfigure}[t]{0.2\textwidth}
        \centering
        \resizebox{\textwidth}{!}{ \def\svgwidth{120pt} \begingroup%
  \makeatletter%
  \providecommand\color[2][]{%
    \errmessage{(Inkscape) Color is used for the text in Inkscape, but the package 'color.sty' is not loaded}%
    \renewcommand\color[2][]{}%
  }%
  \providecommand\transparent[1]{%
    \errmessage{(Inkscape) Transparency is used (non-zero) for the text in Inkscape, but the package 'transparent.sty' is not loaded}%
    \renewcommand\transparent[1]{}%
  }%
  \providecommand\rotatebox[2]{#2}%
  \ifx\svgwidth\undefined%
    \setlength{\unitlength}{208.54558441bp}%
    \ifx\svgscale\undefined%
      \relax%
    \else%
      \setlength{\unitlength}{\unitlength * \real{\svgscale}}%
    \fi%
  \else%
    \setlength{\unitlength}{\svgwidth}%
  \fi%
  \global\let\svgwidth\undefined%
  \global\let\svgscale\undefined%
  \makeatother%
  \begin{picture}(1,1.27741862)%
    \put(0,0){\includegraphics[width=\unitlength,page=1]{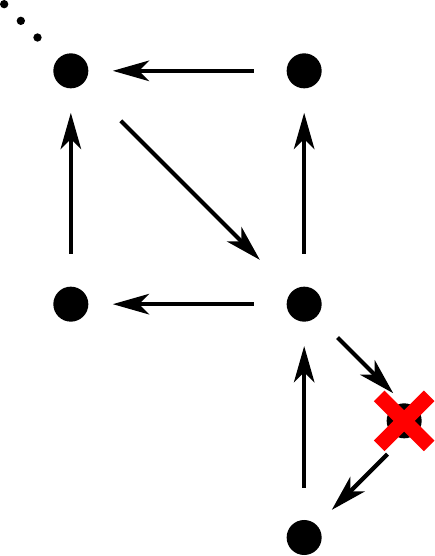}}%
  \end{picture}%
\endgroup%
}

        \caption{If the final vertex lies along an edge.\label{fig:squaresend1}}
    \end{subfigure}%
\hspace{0.05\textwidth}
\begin{subfigure}[t]{0.2\textwidth}
        \centering
        \resizebox{\textwidth}{!}{\def\svgwidth{120pt}  \begingroup%
  \makeatletter%
  \providecommand\color[2][]{%
    \errmessage{(Inkscape) Color is used for the text in Inkscape, but the package 'color.sty' is not loaded}%
    \renewcommand\color[2][]{}%
  }%
  \providecommand\transparent[1]{%
    \errmessage{(Inkscape) Transparency is used (non-zero) for the text in Inkscape, but the package 'transparent.sty' is not loaded}%
    \renewcommand\transparent[1]{}%
  }%
  \providecommand\rotatebox[2]{#2}%
  \ifx\svgwidth\undefined%
    \setlength{\unitlength}{208.54558441bp}%
    \ifx\svgscale\undefined%
      \relax%
    \else%
      \setlength{\unitlength}{\unitlength * \real{\svgscale}}%
    \fi%
  \else%
    \setlength{\unitlength}{\svgwidth}%
  \fi%
  \global\let\svgwidth\undefined%
  \global\let\svgscale\undefined%
  \makeatother%
  \begin{picture}(1,1.03836093)%
    \put(0,0){\includegraphics[width=\unitlength,page=1]{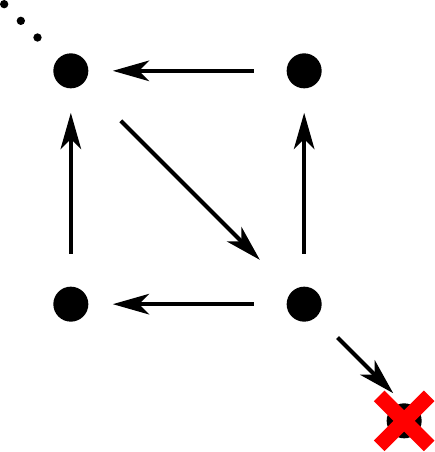}}%
  \end{picture}%
\endgroup%
}

        \caption{If the final vertex is on an outside corner.\label{fig:squaresend2}}
    \end{subfigure}%
\hspace{0.05\textwidth}
\begin{subfigure}[t]{0.2\textwidth}
        \centering
        \resizebox{\textwidth}{!}{\def\svgwidth{120pt}  \begingroup%
  \makeatletter%
  \providecommand\color[2][]{%
    \errmessage{(Inkscape) Color is used for the text in Inkscape, but the package 'color.sty' is not loaded}%
    \renewcommand\color[2][]{}%
  }%
  \providecommand\transparent[1]{%
    \errmessage{(Inkscape) Transparency is used (non-zero) for the text in Inkscape, but the package 'transparent.sty' is not loaded}%
    \renewcommand\transparent[1]{}%
  }%
  \providecommand\rotatebox[2]{#2}%
  \ifx\svgwidth\undefined%
    \setlength{\unitlength}{266.4bp}%
    \ifx\svgscale\undefined%
      \relax%
    \else%
      \setlength{\unitlength}{\unitlength * \real{\svgscale}}%
    \fi%
  \else%
    \setlength{\unitlength}{\svgwidth}%
  \fi%
  \global\let\svgwidth\undefined%
  \global\let\svgscale\undefined%
  \makeatother%
  \begin{picture}(1,1.00000005)%
    \put(0,0){\includegraphics[width=\unitlength,page=1]{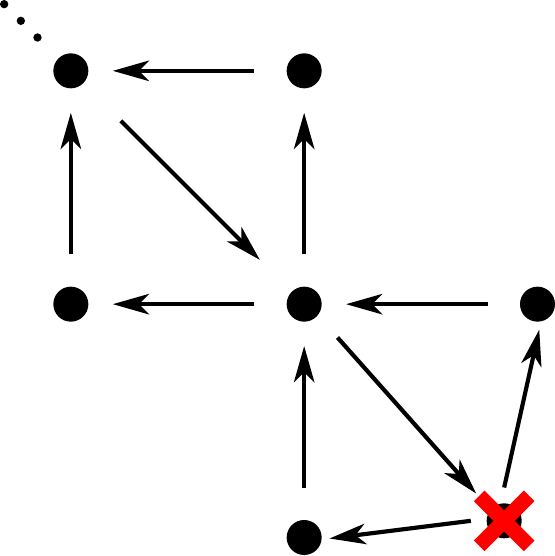}}%
  \end{picture}%
\endgroup%
}
        \caption{If the final vertex is on an inside corner.\label{fig:squaresend3}}
    \end{subfigure}%

  \caption{The three possible ending states of a sequence of mutations along a ``chain of squares,'' as described in the proof of Theorem \ref{mainthm}.}
\end{figure}

At the end of the chain of squares, if the final vertex was on the bottom or right edge it will have a single incoming edge from the next vertex along that edge. If the final vertex is an outside corner of the diagram --- that is, there are no vertices to its south or east --- there is no such vertex. (Note that in both cases we are disregarding the frozen vertices.) Finally, if the final vertex is an inside corner, so there is a vertex to the south and east but not to the southeast, then there are two such vertices. In the first case, at the end of the sequence of mutations, we are left with the picture in Figure \ref{fig:squaresend1}, in the second we are left with the picture in Figure \ref{fig:squaresend2}, and in the third we have Figure \ref{fig:squaresend3}. In all of these cases, we delete the vertex marked with an X, and we are left with the same local picture as before the sequence of mutations began. Therefore, the final result of this process is that $v'$ and $v_{i,j}$ have been merged, and the rest of the diagram is as it was before. This completes the proof.
\end{proof}

\end{document}